\documentclass[11pt,a4paper,oneside,reqno]{amsbook}
\usepackage[utf8]{inputenc}
\usepackage{hyperref}
\usepackage{url}
\usepackage{mathrsfs}

\title[Harnack Inequality]{Harnack-type inequalities for nonlinear evolution equations}
\author{Jessica Slegers}
\date{November 2020}

\usepackage{amsmath,amssymb}
\usepackage{enumerate,enumitem}
\setlist[enumerate]{itemsep=6pt}
\setenumerate{label=(\roman*)}

\usepackage{etoolbox}
\makeatletter
\patchcmd{\@maketitle}{\newpage}{}{}{} 
\makeatother

\usepackage{titletoc}
\titlecontents{section}[0em]{\hspace{12pt}}{\thecontentslabel\hspace{0.5em}}{}{\mdseries\titlerule*[1em]{.}\contentspage}[\smallskip]

\usepackage{titlesec}

\titleformat{\section}{\centering\scshape}{\thesection.}{6pt}{}
\titleformat{\subsection}{\bfseries}{\thesubsection}{6pt}{}

\renewcommand{\thesection}{\arabic{chapter}.\arabic{section}}

\makeatletter
\renewcommand{\@makeschapterhead}[1]{%
  \vspace*{0\p@}%
  {\parindent \z@ \raggedright
    \normalfont
    \interlinepenalty\@M
    \Large \centering \bfseries  #1\par\nobreak
    \vskip 20\p@
  }}
\makeatother

\usepackage{amsthm}

\newtheoremstyle{theorem}{}{}{\itshape}{}{\bfseries}{.}{ }{}
\theoremstyle{theorem}
\newtheorem{theorem}{Theorem}[chapter]
\newtheorem{proposition}[theorem]{Proposition}
\newtheorem{lemma}[theorem]{Lemma}
\newtheorem{corollary}[theorem]{Corollary}

\newtheoremstyle{definition}{}{}{}{}{\bfseries}{.}{ }{}
\theoremstyle{definition}
\newtheorem{definition}{Definition}[chapter]
\newtheorem{remark}{Remark}[chapter]

\numberwithin{equation}{chapter} 
\usepackage{comment} 

\newcommand{\R}{\mathbb{R}}
\newcommand{\N}{\mathbb{N}}
\DeclareMathOperator*{\divergence}{div}
\DeclareMathOperator*{\diam}{diam}
\DeclareMathOperator*{\supp}{supp}
\DeclareMathOperator{\loc}{loc}

\usepackage{graphicx}

\begin{document}
\frontmatter
\maketitle

\begin{center}A thesis submitted in partial fulfilment of\\ the requirements for the degree of\\ B.Sc. (Honours)

\vspace{9.5cm}

\includegraphics[height=3cm]{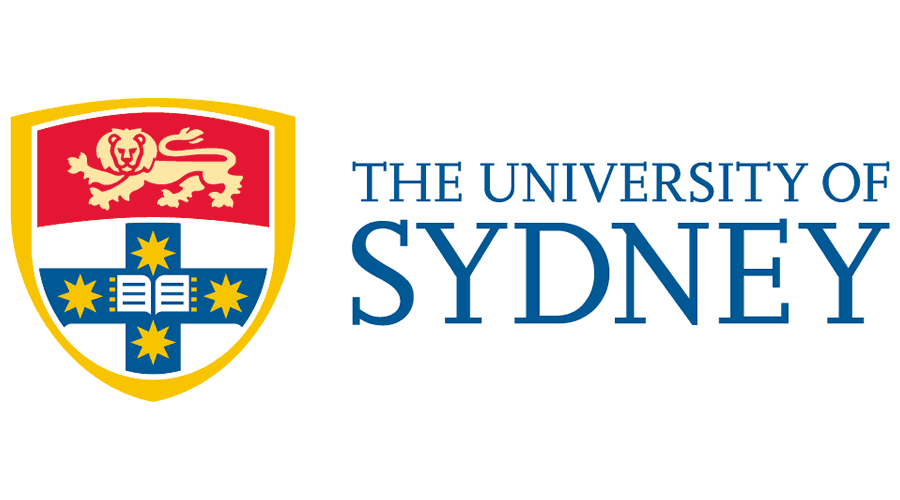}

\vspace{0.5cm}

School of Mathematics and Statistics

University of Sydney

Australia

\vspace{6pt}

November 2021

\end{center}
\chapter*{Preface}

Harnack inequalities are useful qualitative tools for understanding the properties of partial differential equations. Originally discovered as a property of harmonic functions, Harnack inequalities have since been studied for solutions of wider classes of elliptic and parabolic problems.

In this monograph, we take particular interest in deriving Harnack inequalities for solutions of nonlinear evolution equations. We focus on exploring the methods introduced by Li and Yau in the case of the linear heat equation and later extended to nonlinear problems by Auchmuty and Bao. Prior to presenting these results, we study a minimisation problem, which appears naturally in the proofs.

After establishing a family of three general Harnack inequality results by Auchmuty and Bao, we investigate applications to deriving Harnack inequalities satisfied by solutions of the porous medium equation and weak solutions of the parabolic problem associated with the $p$-Laplace operator, which we refer to here as the $p$-diffusion equation. 

Finally, we demonstrate a common application of Harnack inequalities by proving the local space-time H\"older continuity of solutions to a class of linear evolution problems. The proof is based on methods introduced by Moser during his seminal work on Harnack inequalities during the 1960s.

We conclude by suggesting potential opportunities for future work following on from the topics discussed here.

\vspace{1.5cm}

{\let\clearpage\relax \chapter*{Acknowledgements}}

I would like to express my deepest gratitude to my supervisor and mentor Dr. Daniel Hauer for his generous support, encouragement, and kindness. His insights are incredibly valuable to me and have helped me to see mathematics and life more generally in new ways. His unbounded enthusiasm for mathematics has helped to keep me motivated during what has been a challenging, but rewarding year. I thank him for his warm welcome to the world of mathematical research and I look forward to our continued collaboration in future projects.

I would also like to thank my fellow Honours students for the support we have given each other throughout the year. It has been wonderful getting to know you all.

I also wish to thank my mother and father for their ongoing support throughout my education.

Finally, I would like to express my appreciation for the University of Sydney for its award of an Honours Scholarship, which has supported my studies this year.

\tableofcontents

\newpage

\mainmatter
\chapter{Introduction}\label{sec:intro}
When studying partial differential equations, it is important to develop qualitative and quantitative tools, such as inequalities, which can be used to discover features of their solutions, for example, their H\"older continuity or maximum principles. 

In 1887, Harnack \cite{Harnack1887} proved the following inequality for positive harmonic functions defined on the open ball $B_R(x_0) \subseteq \R^2$:
\begin{equation*}\label{eq:OGHarnack}
    \frac{R-r}{R+r}u(x_0) \leq u(x) \leq \frac{R+r}{R-r}u(x_0).
\end{equation*}
This inequality holds for all $x \in B_r(x_0)$ and $r<R$. We note that throughout this monograph, a function $u$ defined on $\Omega$ will be called \emph{positive} if $u(x)>0$ for a.e. $x \in \Omega$ and \emph{nonnegative} if $u(x) \geq 0$ for a.e. $x \in \Omega$. An analogous formulation of this result may also be given when $u$ is a positive harmonic function on a domain $\Omega \subseteq \R^d$ for $d \geq 1$. Often, Harnack's inequality is stated in the following form
\begin{equation}\label{eq:meanvalue}
    \gamma^{-1} \sup_{B_r(x_0)}u \leq u(x_0) \leq \gamma \inf_{B_r(x_0)}u,
\end{equation} which is due to Kellogg \cite{Kellogg1929}. Here, we assume that $\bar{B}_r(x_0) \subseteq \Omega$ and importantly, the constant $\gamma > 1$ is independent of the function $u$. Alternatively, the ball $B_r(x_0)$ may be replaced with any subdomain $\Omega'$ compactly contained in $\Omega$, meaning that $\overline{\Omega'} \subseteq \Omega$.

Harnack's inequality has several important consequences in the theory of harmonic functions.
\begin{enumerate}
    \item (Liouville's Theorem, \cite{Evans}) Every nonnegative harmonic function on $\R^d$ is constant.
    \item (Removable Singularity Theorem, \cite{Axler}) Let $d \geq 3$. If a function $u:B_r(0)\setminus \{0\} \rightarrow \R$ is harmonic and $u(x) = o(|x|^{2-d})$ for $x \rightarrow 0$, then $u(0)$ can be defined so that $u:B_r(0) \rightarrow \R$ is a harmonic function. 
    \item (Harnack's First Convergence Theorem, \cite{GilbargTrudinger}) If a sequence of harmonic functions on a bounded domain $\Omega \subset \R^d$, which are continuous on $\bar{\Omega}$, converges uniformly on the boundary $\partial \Omega$, then the sequence converges uniformly in $\Omega$ to a harmonic function.
    \item (Harnack's Second Convergence Theorem, \cite{GilbargTrudinger}) If a sequence $(u_n)$ of harmonic functions on a connected open set $\Omega \subset \R^d$ is monotone increasing and there exists $x_0 \in \Omega$ such that $(u_n(x_0))$ is convergent, then $(u_n)$ converges uniformly on every compact subset of $\Omega$ to a harmonic function $u$.
\end{enumerate}

Following Harnack's discovery, attention has been given to proving similar inequalities for parabolic equations. In the parabolic case, a time dependence can be seen in the resulting inequalities. In 1954, Hadamard \cite{Hadamard1954} and Pini \cite{Pini1954} independently derived Harnack-type inequalities of the form
\begin{equation}\label{eq:paraharnack}
    \gamma^{-1} \sup_{x \in B_r(x_0)}u(x, t_0-r^2) \leq u(x_0, t_0) \leq \gamma \inf_{x \in B_r(x_0)}u(x, t_0+r^2)
\end{equation}
for positive solutions to the heat equation, 
\begin{equation}\label{eq:heat}
u_t = \Delta u \qquad \text{in }\Omega \times (0,\infty).
\end{equation}
In inequality \eqref{eq:paraharnack}, we see a waiting time from $t_0 - r^2$ to $t_0$ and from $t_0$ to $t_0 + r^2$. Using the physical interpretation of \eqref{eq:heat} as describing the distribution of heat in a medium, this means that if the temperature $u$ is known at a point $(x_0,t_0)$, then after waiting some time, the temperature everywhere in a neighbourhood of $x_0$ will be at least $\frac{u(x_0,t_0)}{\gamma}$. This waiting time is in fact necessary for the result to hold, as has been demonstrated using counterexamples, such as the following one contributed by Moser \cite{Moser1964}.

Suppose there exists $\gamma >1$ independent of $u$ such that \begin{equation}\label{eq:contradiction}
    \sup_{x \in K} u(x,t_0) \leq \gamma \inf_{x \in K}u(x,t_0)
\end{equation}
for all solutions $u$ of \begin{equation}u_t = u_{xx}\qquad \text{in } \R \times (0,\infty),\label{eq:onedimheat}\end{equation} where $K \subset \subset \R$ and $t_0>0$ is fixed. The function $$u_{\xi}(x,t):=t^{-1/2}e^{-(x+\xi)^2/4t}$$ is known to be a solution of \eqref{eq:onedimheat} for all $\xi \in \R$. Observe that for any $x_0>0$ fixed, one has that $$\lim_{\xi \rightarrow -\infty} \frac{u_{\xi}(0,1)}{u_{\xi}(x_0,1)} = \lim_{\xi \rightarrow -\infty} e^{x_0^2/4}e^{x_0\xi/4}=0.$$ Applying \eqref{eq:contradiction} for some $K \subset \subset \R$ containing $x_0$ and $0$ yields $$\left(\frac{u_{\xi}(0,1)}{u_{\xi}(x_0,1)} \right)^{-1} \leq \gamma.$$ Taking $\xi \rightarrow -\infty$, we reach a contradiction, since the left-hand side becomes unbounded in the limit.

In 1964, Moser \cite{Moser1964} also expanded on this work by obtaining a similar result for weak solutions of more general parabolic differential equations with variable coefficients, that is, equations in the following divergence form
\begin{equation}\label{eq:para}
    \frac{\partial u}{\partial t} = \sum_{k,l = 1}^{n} \frac{\partial}{\partial x_k}\left( a_{kl}(x,t)\frac{\partial u}{\partial x_l}\right) \qquad \text{in }\Omega \times (0,T).
\end{equation}
Here, the coefficients $a_{kl}(x,t)$ are measurable functions with $a_{kl} = a_{lk}$ that satisfy the uniform ellipticity condition \begin{equation}\label{eq:ellipticity}\lambda^{-1} \leq \sum_{k,l = 1}^{n} a_{kl}(x,t) \xi_k\xi_l \leq \lambda\end{equation} for all $\xi \in \R^d$ with $\sum_k \xi^2_k=1$. In particular, Moser proved that there exists a positive constant $C$ such that \begin{equation*}\label{eq:cylinder}
    \sup_{B_r(x_0) \times (t_1^-,t_2^-)} u\leq C \inf_{B_r(x_0) \times (t_1^+,t_2^+)} u
\end{equation*} for all nonnegative weak solutions $u$ of \eqref{eq:para} where $r>0$ is such that $\bar{B}_r(x_0) \subseteq \Omega$ and
$0< t_1^- < t_2^- < t_1^+ < t_2^+ < T$. Here, one must understand the supremum and infimum in the sense of the essential supremum and essential infimum respectively. Moser then used the oscillation of weak solutions of \eqref{eq:para} to prove that these are space-time H\"older continuous on the interior of a rectangular domain. Although H\"older continuity had already been established by Nash \cite{Nash1958} in 1958, Moser presented a new method for proving this result, which we will explore in Chapter \ref{sec:hoelder}. 

Research in pursuit of Harnack-type inequalities has been heavily influenced by the work of Li and Yau \cite{LiYau1986}, who developed new methods for deriving such results. In their analysis of positive solutions of the heat equation \eqref{eq:heat} on $\R^d \times (0,T)$, the sharp gradient estimate
\begin{equation}
\label{eq:liyauestimate}
\frac{|\nabla u|^2}{u^2} - \frac{u_t}{u} \leq \frac{d}{2t}
\end{equation}
was found. By integrating this inequality over a path connecting $(x_1,t_1)$ and $(x_2, t_2)$ in $\R^d \times (0,T)$ with $0 < t_1 < t_2<T$, it follows that for all such points
\begin{equation}\label{eq:heatineq}
    u(x_2,t_2) \geq u(x_1,t_1) \left(\frac{t_1}{t_2} \right)^{d/2} e^{-\tfrac{|x_2-x_1|^2}{4(t_2-t_1)}}.
\end{equation}
We remark that Li and Yau's results have particular significance, since they hold more generally on complete Riemannian manifolds with nonnegative Ricci curvature. However, the exploration of Harnack inequalities on manifolds is beyond the scope of this monograph and we instead restrict our attention to the Euclidean space $\R^d$.

Harnack-type estimates analogous to \eqref{eq:paraharnack} have also been formulated and proven for parabolic and elliptic equations in non-divergence form, for example, by Krylov and Safonov \cite{KrylovSafonov1981,Safonov1983}.  Similar to the case of operators in divergence form, this result can be used to prove the H\"older continuity of the solutions.

Another topic of interest is the notion of a \emph{boundary Harnack inequality}, which was first introduced by Kemper \cite{Kemper1972} in 1972 for harmonic functions and solutions of the heat equation. We now recall one of the main results from this paper. Let $\Omega$ be a bounded domain in $\R^d$ and $\Gamma$ a compact subset of the boundary $\partial \Omega$. Let $\Omega' \subset \Omega$ be a domain such that $\partial \Omega' \cap \partial \Omega$ is compactly contained in the interior of $\Gamma$. Then for $x_0 \in \Omega'$, there exists a constant $C>0$ such that
\begin{equation*}\label{eq:KemperBoundary}
    \sup_{x \in \Omega'} u(x) \leq Cu(x_0)
\end{equation*} for every nonnegative harmonic function $u$ that vanishes on $\Gamma$. Unlike the case of inequality \eqref{eq:meanvalue}, it does not make sense for a similar two-sided inequality to be obtained here. Indeed, the assumption that $u$ vanishes on $\Gamma$ guarantees that the infimum of $u$ will be zero. Thus, any attempt to control the supremum by the infimum in the usual manner would only be valid for $u \equiv 0$.

\section{Modelling}\label{sec:modelling}
Throughout this monograph, we are interested in studying some particular parabolic equations, specifically the linear heat equation, the porous medium equation, and the $p$-diffusion equation. Each of these equations arise naturally in physical phenomena and will be introduced below.

First considering the porous medium equation, let $M>1$ and consider $u$ to be a scalar-valued function representing the density of a gas flowing in a porous medium $\Omega \subseteq \R^d$. The evolution of the density $u$ is governed by the continuity equation 
\begin{equation}\label{eq:PMEcont}
    u_t + \divergence (uV) =0 \qquad \text{in } \Omega \times [0,\infty)
\end{equation}
where $V$ denotes the velocity of the gas \cite{Esteban-Vazquez-1988}. Assuming that the flow is laminar, the velocity $V$ is related to the pressure $f$ of the gas by the linear Darcy law $V = -\nabla f$ and the pressure $f$ is proportional to $u^{M-1}$. Using a rescaling argument if necessary, we choose the constant of proportionality to be $\frac{M}{M-1}$. Combining this information yields the relationship
$$V=-\nabla \big( (\tfrac{M}{M-1})u^{M-1}\big).$$
By a direct calculation, it can be seen that $uV = -\nabla (u^M)$. Inserting this into \eqref{eq:PMEcont}, we arrive at the porous medium equation
\begin{equation*}
    u_t = \Delta (u^M) \qquad \text{in } \Omega \times [0,\infty).
\end{equation*}

Next, seeking to understand the $p$-diffusion equation, we suppose that $u$ represents the electric potential in a medium $\Omega \subseteq \R^d$. Let $\sigma: \Omega \rightarrow [0,\infty]$ be a measureable function corresponding to the conductivity of the medium $\Omega$. We follow the explanation provided in \cite{Brander}. Ohm's law states that the electric current density $J$ obeys $J=-\sigma \nabla u$. Combining this with Kirchhoff's law, $\divergence(J)=0$, we obtain the linear conductivity equation
\begin{equation}\label{eq:linear conductivity equation}
    \divergence(\sigma \nabla u)=0 \qquad \text{in } \Omega.
\end{equation}
Alternatively, using the continuity equation $u_t + \divergence(J)=0$ instead of Kirchhoff's law produces the associated parabolic equation
\begin{equation}\label{eq:parabolic linear conductivity equation}
    u_t = \divergence(\sigma \nabla u) \qquad \text{in } \Omega \times [0,\infty).
\end{equation}
In the case $\sigma \equiv 1$ corresponding to a medium of uniform conductivity, we recover from \eqref{eq:linear conductivity equation} and \eqref{eq:parabolic linear conductivity equation} the usual Laplace equation $\Delta u = 0$ as well as the linear heat equation $u_t = \Delta u$, best known as a model of heat diffusion. 

Several media encountered in the physical world do not obey the linear Ohm's law. One possible alternative is that the current density $J$ satisfies a power law $J =-\sigma |\nabla u|^{p-2}\nabla u$ for some $1 < p < \infty$. Consequently, by Kirchhoff's law, $u$ instead satisfies
\begin{equation*}
    \divergence(\sigma |\nabla u|^{p-2}\nabla u)= 0 \qquad \text{in } \Omega.
\end{equation*}
Again, by setting $\sigma \equiv 1$, we recover the $p$-Laplace equation
\begin{equation*}
    \Delta_p u:= \divergence(|\nabla u|^{p-2}\nabla u)=0 \qquad \text{in } \Omega.
\end{equation*}
Using a continuity equation, the related parabolic equation, which we call the $p$-diffusion equation, can be derived as
\begin{equation*}
    u_t = \Delta_p u \qquad \text{in } \Omega \times [0,\infty).
\end{equation*}

\section{Main results}\label{sec:main-results}

We begin by reviewing some  Harnack-type inequalities discovered by Auchmuty and Bao \cite{Auchmuty-Bao-1994}, whose methods were heavily inspired by the work of Li and Yau. In particular, the ideas of Li and Yau were adapted to derive Harnack inequalities for solutions to parabolic equations satisfying a more general gradient estimate.

\begin{theorem}[General Harnack inequalities, \cite{Auchmuty-Bao-1994}]\label{thm:generalharnackineq}
Let $\Omega \subseteq \R^d$ be an open convex set and let $f$ be a positive continuously differentiable function on $\Omega_T:= \Omega \times (0,T)$, for which there exist constants $C>0$, $p>1$, $r \in \R$, and a function $a \in L_{\loc}^1(0,T)$ such that 
\begin{equation}\label{eq:thm1}
    \frac{\partial f}{\partial t} + af \geq \frac{C|\nabla f|^p}{f^r} \qquad \text{in } \Omega_T.
\end{equation}
Then for all $x_1,x_2 \in \Omega$ and $0 < t_1 < t_2 <T$:
\begin{enumerate}[label=({\roman*})]
    \item if $r=p-1$, then
    \begin{equation*}\label{eq:casei}
        f(x_2,t_2)\geq \frac{e^{A(t_1)}}{e^{A(t_2)}}f(x_1,t_1)e^{-\frac{\xi|x_2-x_1|^q}{(t_2-t_1)^{q-1}}};
    \end{equation*}
    \item if $r>p-1$, then
    \begin{equation*}\label{eq:caseii}\begin{split}
        f(x_2,t_2)&\geq \frac{e^{A(t_1)}}{e^{A(t_2)}}f(x_1,t_1)\times \\&\qquad (1+m\xi|x_2-x_1|^qI^{1-q}[f(x_1,t_1)]^m e^{mA(t_1)})^{\frac{-1}{m}};\end{split}
    \end{equation*}
    \item if $r<p-1$, then
    \begin{equation}\label{eq:caseiiia}
        [f(x_2,t_2)]^{|m|}\geq \left( \frac{e^{A(t_1)}}{e^{A(t_2)}} \right)^{|m|} \left( [f(x_1,t_1)]^{|m|} - \frac{|m| \, \xi \, |x_2-x_1|^qI^{1-q}}{e^{|m|A(t_1)}} \right).
    \end{equation}
    If the quantity in the large parentheses above is nonnegative, then
    \begin{equation*}\label{eq:caseiiib}
        f(x_2,t_2)\geq \frac{e^{A(t_1)}}{e^{A(t_2)}}f(x_1,t_1) \left( 1 - \frac{|m| \, \xi \, |x_2-x_1|^qI^{1-q}}{[f(x_1,t_1)]^{|m|}e^{|m|A(t_1)}} \right)^{\frac{1}{|m|}},
    \end{equation*}
\end{enumerate}
where $q:= \frac{p}{p-1}$, $m:=\frac{r}{p-1}-1$, $\xi := \frac{1}{q}(\frac{1}{pC})^{q-1}$, $A(t)$ is an antiderivative of $a(t)$, and $I:=\int_{t_1}^{t_2}e^{m(p-1)A(t)} \ \mathrm{d}t$.

In addition, if $f$ is nonnegative on $\Omega_T$ and satisfies inequality \eqref{eq:thm1} with $r=0$, then \eqref{eq:caseiiia} holds for all $x_1,x_2 \in \Omega$ and $0 < t_1 < t_2 <T$.
\end{theorem}

The key idea in the proof of this theorem is to connect a pair of points $(x_1,t_1), (x_2,t_2) \in \Omega_T$ by an arbitrary continuously differentiable path $x(t)$. By manipulating the gradient estimate \eqref{eq:thm1}, an estimate of the time derivative of the function $f$ along the path $x(t)$ may be found. After integrating this new inequality over an optimally chosen path $x(t)$ and some rearrangement, the final results of Theorem \ref{thm:generalharnackineq} may be obtained. This proof will be discussed in further detail in Chapter \ref{sec:genharnack}, along with a second version of this result posed under weaker assumptions (Theorem \ref{thm:gen-dist}). We also present another similar result obtained from a different gradient estimate (Theorem \ref{thm:gen-backwards}).

As applications of Theorem \ref{thm:generalharnackineq}, we derive Harnack inequalities for solutions to three significant parabolic equations, which may all be found in \cite{Auchmuty-Bao-1994}. The first equation we consider is the linear heat equation \eqref{eq:heat}, the result for which was stated earlier in the introduction, but we formulate here more precisely.
\begin{theorem}\label{thm:heat}
Let $u$ be a positive solution of the heat equation
\begin{equation*}
    u_t = \Delta u \qquad \text{in } \R^d \times (0,T).
\end{equation*}
Then $u$ satisfies
\begin{equation*}
    u(x_2,t_2) \geq u(x_1,t_1) \left(\frac{t_1}{t_2} \right)^{d/2} e^{-\tfrac{|x_2-x_1|^2}{4(t_2-t_1)}}.
\end{equation*}
for all $x_1,x_2 \in \R^d$ and $0 < t_1 < t_2 < T$.
\end{theorem}
In combination with \eqref{eq:liyauestimate}, this result may be recovered from Theorem \ref{thm:generalharnackineq} with $a(t):= \frac{d}{2t}$, $C=1$, $p=2$, and $r=1$.

We also consider two nonlinear variants of the heat equation, namely the porous medium equation and the $p$-diffusion equation. The proof in both cases relies on applying an Aronson-B\'enilan-type estimate to arrive at a gradient estimate of the form \eqref{eq:thm1}. The result in the case of the porous medium equation can then be obtained from Theorems \ref{thm:generalharnackineq} and \ref{thm:gen-backwards} with $a(t):= \frac{M-1}{M-1+\frac{2}{d}} \frac{1}{t}$, $C=1$, $p=2$ and $r=0$.

\begin{theorem}\label{thm:pme}
Let $M> M_0(d):=\max\{ 0, 1-\frac{2}{d}\}$ and let $u$ be a positive solution to the porous medium equation
\begin{equation*}\label{eq:PME}
    u_t = \Delta (u^M) \qquad \text{in } \R^d \times (0,T).
\end{equation*} Then, for all $x_1,x_2 \in \R^d$ and $0<t_1<t_2<T$, one has that
\begin{enumerate}
\item if $M>1$, then
\begin{equation*}
    \hspace{1.2cm}[u(x_2,t_2)]^{M-1} \geq \left( \frac{t_1}{t_2} \right)^{\mu}\left[[u(x_1,t_1)]^{M-1} - \frac{M-1}{M}\frac{\delta|x_2-x_1|^2}{4(t_2^{\delta} - t_1^{\delta})}\frac{1}{t_1^{\mu}}\right];
\end{equation*}
\item if $M=1$, then the porous medium equation reduces to the heat equation \eqref{eq:heat} and the result of Theorem \ref{thm:heat} holds;
\item if $M_0(d) < M < 1$, then 
\begin{equation*}
    \hspace{1.5cm}[u(x_2,t_2)]^{1-M} \geq \left( \frac{t_2}{t_1} \right)^{\mu}\left[[u(x_1,t_1)]^{M-1} - \frac{M-1}{M}\frac{\delta|x_2-x_1|^2}{4(t_2^{\delta} - t_1^{\delta})}\frac{1}{t_1^{\mu}}\right]^{-1},\label{eq:harnackPME-v2}
\end{equation*}
\end{enumerate}
where $\mu:=\frac{M-1}{M-1+\frac{2}{d}}$ and $\delta :=1-\mu$.
\item 
\end{theorem}

Finally we obtain a result for the $p$-diffusion equation by using Theorems \ref{thm:gen-dist} and \ref{thm:gen-backwards} with $a(t)=\frac{\gamma K}{t}$, $C=1$ and $r=0$. Here $\gamma:=\frac{p-2}{p-1}$ and $K$ is a constant obtained from a relevant Aronson-B\'enilan-type estimate.
\begin{theorem}\label{thm:pdiff}
Let $\frac{2d}{d+1} < p < \infty$ and let $u$ be a positive solution of the $p$-diffusion equation,
\begin{equation*}\label{eq:pdiff}
    \frac{\partial u}{\partial t} = \divergence(|\nabla u|^{p-2}\nabla u) \qquad \text{in }  \R^d \times (0,T).
\end{equation*}
Then, for all $x_1,x_2 \in \R^d$ and $0 < t_1 < t_2 <T$, one has that
\begin{enumerate}
    \item if $p>2$, then 
    \begin{equation*}\label{eq:pdiffharnack}
    \hspace{1cm}[u(x_2,t_2)]^{\gamma} \geq \left( \frac{t_1}{t_2} \right)^{\gamma K} \left( [u(x_1,t_1)]^{\gamma} -\gamma \xi|x_2-x_1|^qI^{1-q}t_1^{-\gamma K} \right);
\end{equation*}
    \item if $p=2$, then the $p$-diffusion equation reduces to the heat equation \eqref{eq:heat} and the result of Theorem \ref{thm:heat} holds;
    \item if $\frac{2d}{d+1} < p < 2$, then 
    \begin{equation*}\label{eq:pdiff-lower-p}
        \hspace{1.4cm}[u(x_2,t_2)]^{-\gamma} \geq \left( \frac{t_2}{t_1} \right)^{\gamma K} \left( [u(x_1,t_1)]^{\gamma} -\gamma \xi|x_2-x_1|^qI^{1-q}t_1^{-\gamma K} \right)^{-1},
    \end{equation*}
\end{enumerate}
where $\gamma:=\frac{p-2}{p-1}$, $q:=\frac{p}{p-1}$, $\xi:=\frac{1}{q} \left(\frac{1}{p} \right)^{q-1}$, $\delta:=(2-p)K+1$, and 
$$I:=\begin{cases}
\frac{t_2^{\delta} - t_1^{\delta}}{\delta} & \delta \ne 0, \\
\log t_2 - \log t_1 & \delta=0.
\end{cases}$$
\end{theorem}

\chapter{Preliminaries}\label{sec:min}

\section{Minimisation of a Convex Functional}
During the proof of the general Harnack inequality results of Auchmuty and Bao \cite{Auchmuty-Bao-1994} in Chapter \ref{sec:genharnack}, our efforts to optimise the bounds in the inequalities will lead us to solving a minimisation problem. We discuss the details of this variational problem in this chapter.

Let $1 < q < \infty$, $q'=\frac{q}{q-1}$ be the conjugate exponent of $q$, and $0<t_1<t_2$. Let $w:[t_1,t_2]\rightarrow (0,\infty)$ be continuous. In the following, let $\|\cdot\|_q$ denote the usual norm on $L^q(t_1,t_2;\R^d)$ and let 
$$\|x\|_{q,w}:= \left( \int_{t_1}^{t_2}|x(t)|^qw(t)\ \mathrm{d}t\right)^{1/q} $$ be the weighted norm corresponding to the measure $d\mu(t) = w(t) \ \mathrm{d}t$ for the weighted Lebesgue space, which we denote by $L^q_w(t_1,t_2;\R^d)$. 

Then, denote by $W^{1,q}_w(t_1,t_2;\R^d)$ the weighted Sobolev space, which we understand as the space of functions $x \in L^q_w(t_1,t_2;\R^d)$ with  weak derivative also contained in $L^q_w(t_1,t_2;\R^d)$. We equip the space $W^{1,q}_w(t_1,t_2;\R^d)$ with a weighted Sobolev norm given by 
$$\|x\|_{W^{1,q}_{w}} :=\|x\|_{q,w} + \|\dot{x}\|_{q,w}.$$
Here, we understand $\dot{x}$ as the weak derivative of $x \in W^{1,q}_w(t_1,t_2;\R^d)$. In addition, we define the space $W^{1,q}_{w,0}(t_1,t_2;\R^d)$ as the closure of the test functions $C^{\infty}_c(t_1,t_2;\R^d)$ in $W^{1,q}_w(t_1,t_2;\R^d)$. For some basic properties of Sobolev spaces, we direct the reader to Appendix \ref{sobolev}, in which we provide a presentation of the usual unweighted Sobolev spaces, that is, $W^{1,q}_w(t_1,t_2;\R^d)$ with $w\equiv 1$. However, any important properties we require hold true for the weighted spaces used in this chapter.

Define a functional $E:W^{1,q}_w(t_1,t_2;\R^d) \rightarrow \R$ by
\begin{equation*}
    E(x):= \frac{1}{q}\int_{t_1}^{t_2}|\dot{x}(t)|^q w(t)\ \mathrm{d}t
\end{equation*} for all $x\in W^{1,q}_w(t_1,t_2;\R^d)$. We aim to minimise the functional $E$ over the affine space 
$$\mathcal{A}:= \{\xi_{0}\} \oplus W^{1,q}_{w,0}(t_1,t_2;\R^d),$$ where $\xi_{0}:[t_1,t_2]\rightarrow \R^d$ is a fixed representative of the set of functions $$\{ \xi \in W^{1,q}_w(t_1,t_2;\R^d) \mid \xi(t_1) = x_1, \xi(t_2) = x_2 \}$$ and $x_1,x_2 \in \R^d$ are fixed. For example, one may choose $\xi_{0}$ to be the straight line segment connecting $x_1$ and $x_2$ in $\R^d$. In addition, we note that $\mathcal{A}$ is a closed, convex subset of $W^{1,q}_w(t_1,t_2;\R^d)$.

The main goal of this chapter will be to prove the following result.

\begin{theorem}\label{generalised weight min} Let $1<q<\infty$, $0<t_1<t_2$, and $x_1, x_2 \in \R^d$. 
Let \mbox{$E:W^{1,q}_w(t_1,t_2;\R^d) \rightarrow \R$} be the functional defined by 
\begin{equation}E(x):=\frac{1}{q}\int_{t_1}^{t_2} |\dot{x}(t)|^q w(t) \ \mathrm{d}t,\label{eq:general functional}\end{equation}
where $w:[t_1,t_2]\rightarrow (0,\infty)$ is a continuous function. Then,
$$\min_{x \in \mathcal{A}} E(x) = \frac{|x_2-x_1|^q}{q\big( W(t_2) - W(t_1) \big)^{q-1}},$$
where $W(t):=\int^t \big(w(s)\big)^{\frac{1}{1-q}} \ ds$ is an antiderivative of $\big(w(t)\big)^{\frac{1}{1-q}}$.
\end{theorem}

In practice, it is more convenient to understand the value of $\min_{x \in \mathcal{A}} E(x)$ by the equivalent expression, $$\min_{x \in W^{1,q}_{w,0}(t_1,t_2;\R^d)} E(\xi_0 + x).$$
Hence, we first study properties of the mapping $x \mapsto E(\xi_0 + x)$ as a functional on the space $W^{1,q}_{w,0}(t_1,t_2;\R^d)$, which will aid us in the proof. Throughout, we assume $q>1$ and $E$ is the functional given by \eqref{eq:general functional}.

\begin{proposition}\label{poincare}
The space $W^{1,q}_{w,0}(t_1,t_2;\R^d)$ with the weighted norm $\|\cdot\|_{q,w}$ satisfies a Poincar\'e inequality, that is, there exists $C>0$ such that 
\begin{equation}\label{eq:poincare}
    \|x\|_{q,w} \leq C \|\dot{x}\|_{q,w}
\end{equation}
for all $x \in W^{1,q}_{w,0}(t_1,t_2;\R^d)$.
\end{proposition}

\begin{proof}
We first note that the weighted norm $\|\cdot\|_{q,w}$ is equivalent to the usual norm $\|\cdot\|_q$. The weight function $w$ is continuous on the compact interval $[t_1,t_2]$. Therefore, by the extreme value theorem, $w$ attains a minimum and maximum value on $[t_1,t_2]$. In particular, the positivity of $w$ guarantees both its minimum and maximum values are positive. Therefore, we may define positive constants $C_1$ and $C_2$ such that $$C_1^q:=\min_{t\in [t_1,t_2]}w(t), \qquad  C_2^q:=\max_{t\in [t_1,t_2]}w(t).$$ Then,
$C_1^q \leq w(t) \leq C_2^q$ for all $t \in [t_1,t_2]$. Since $|x(t)|^q \geq 0$, it follows that
$$ C_1^q\int_{t_1}^{t_2}|x(t)|^q \ \mathrm{d}t  \leq \int_{t_1}^{t_2} |x(t)|^q w(t) \ \mathrm{d}t \leq C_2^q \int_{t_1}^{t_2}|x(t)|^q \ \mathrm{d}t$$ for all $x \in W^{1,q}_{w,0}(t_1,t_2;\R^d)$, which implies
$$C_1\|x\|_q \leq \|x\|_{q,w} \leq C_2 \|x\|_q.$$ Thus, the norms $\|\cdot\|_{q,w}$ and $\|\cdot\|_q$ are equivalent as claimed. Therefore, we  derive a Poincar\'e inequality for $\|\cdot\|_{q,w}$ using a similar result for $\|\cdot\|_q$.

Let $x \in W^{1,q}_{w,0}(t_1,t_2;\R^d)$ be arbitrary. By characterisation \ref{w0char} in the Appendix, $x(t_1)=0$. Using this fact and an application of H\"older's inequality, we have that for all $t \in (t_1,t_2)$,
\begin{align*}
    |x(t)| = |x(t)-x(t_1)| &\leq \int_{t_1}^t |\dot{x}(t)|\ \mathrm{d}t\\
    & \leq \left( \int_{t_1}^t |\dot{x}(t)|^q\ \mathrm{d}t\right)^{1/q}|t-t_1|^{1/q'}\\
    & \leq \|\dot{x}\|_q |t-t_1|^{1/q'}
\end{align*}
Therefore, $|x(t)|^q \leq \|\dot{x}\|_q^q |t-t_1|^{q/q'}$, and by integrating over $(t_1,t_2)$,
$$\int_{t_1}^{t_2}|x(t)|^q \ \mathrm{d}t\leq \|\dot{x}\|_q^q \int_{t_1}^{t_2}|t-t_1|^{q/q'} \ \mathrm{d}t.$$
Hence, we arrive at the  Poincar\'e inequality
\begin{equation*}
    \|x\|_q \leq C_3 \|\dot{x}\|_q,
\end{equation*}
for $\|\cdot\|_q$ where $C_3:= \left( \int_{t_1}^{t_2}|t-t_1|^{q/q'} \ \mathrm{d}t\right)^{1/q} = q^{-1/q}(t_2-t_1)$.

The Poincar\'e inequality for $\|\cdot\|_{q,w}$ is an immediate consequence of this result and the equivalence of the norms. Indeed,
$$\|x\|_{q,w} \leq C_2 \|x\|_q \leq C_2C_3\|\dot{x}\|_q \leq \frac{C_2 C_3}{C_1} \|\dot{x}\|_{q,w}$$
and we conclude
$$\|x\|_{q,w} \leq C \|\dot{x}\|_{q,w}$$ for $C:=\frac{C_2 C_3}{C_1}$.
\end{proof}

\begin{proposition}\label{E is coercive}
For a given function $\xi_0 \in W^{1,q}_{w}(t_1,t_2;\R^d)$, the mapping $x \mapsto E(\xi_o + x)$ is coercive on $W^{1,q}_{w,0}(t_1,t_2;\R^d)$.
\end{proposition}
\begin{proof}
Observe that the functional $E$ satisfies $$(qE(\xi_o + x))^{1/q} = \|\dot{\xi_0}+ \dot{x}\|_{q,w}$$ for all $x \in W^{1,q}_{w,0}(t_1,t_2;\R^d)$. Then, using the Poincar\'e inequality \eqref{eq:poincare} with $C$ defined as in the proof of Proposition \ref{poincare},
\begin{align*}
    (qE(\xi_0 + x))^{1/q} &= \|\dot{\xi_0} + \dot{x}\|_{q,w} + \|\dot{\xi_0}\|_{q,w} - \|\dot{\xi_0}\|_{q,w}\\
    &\geq \|\dot{x}\|_{q,w} - \|\dot{\xi_0}\|_{q,w}\\
    &= \frac{1}{C+1}(C\|\dot{x}\|_{q,w} + \|\dot{x}\|_{q,w})- \|\dot{\xi_0}\|_{q,w}\\
    &\geq \frac{1}{C+1}(\|x\|_{q,w}+\|\dot{x}\|_{q,w})- \|\dot{\xi_0}\|_{q,w}\\
    &=\frac{\|x\|_{W^{1,q}_{w}}}{C+1}- \|\dot{\xi_0}\|_{q,w}
\end{align*}
Remembering that $\|\dot{\xi_0}\|_{q,w}$ is a fixed constant, we see that $E(\xi_0 + x) \rightarrow \infty$ as $\|x\|_{W^{1,q}_{w}} \rightarrow \infty$. Thus, $E$ is coercive on $W^{1,q}_{w,0}(t_1,t_2;\R^d)$.
\end{proof}

\begin{proposition}\label{E is convex} The functional $E$ given by \eqref{eq:general functional} is convex and continuous on $W^{1,q}_{w}(t_1,t_2;\R^d)$.
\end{proposition}
\begin{proof}
Seeing that $E(x)$ can be written as $E(x) = \frac{1}{q}\|\dot{x}\|_{q,w}^q$, $E$ is clearly continuous on $W^{1,q}_{w}(t_1,t_2;\R^d)$ by the continuity of the norm $\|\cdot\|_{q,w}$. Thus, we focus on proving the convexity of $E$.

Let $q\geq 1$ and consider the smooth function $g(s)=s^q$ defined for $s\geq0$. Then,
$$\frac{d^2 g}{ds^2} = q(q-1)s^{q-2} \geq 0$$ for all $s \geq 0$ and therefore $g$ is convex. Hence, by the definition of convexity,
$$(\lambda t + (1-\lambda)s)^q \leq \lambda t^q + (1-\lambda)s^q$$ for all nonnegative $s,t$ and $\lambda \in [0,1]$. Also, by the properties of any norm, the Euclidean norm on $\R^d$ is a nonnegative convex function. Indeed, for $x,y \in \R^d$ arbitrary and for all $\lambda \in [0,1]$,
$$ |\lambda x + (1-\lambda)y| \leq|\lambda||x| + |1-\lambda||y|=\lambda|x| + (1-\lambda)|y|$$
Letting $f(x) = |x|^q$, we combine these results to yield that
\begin{align*}
    f(\lambda x + (1-\lambda)y) &= |\lambda x + (1-\lambda)y|^q \\
    &\leq (\lambda|x| + (1-\lambda)|y|)^q \\
    &\leq \lambda|x|^q+(1-\lambda)|y|^q \\
    &= \lambda f(x) + (1-\lambda)f(y)
\end{align*}
and therefore $f(x) =|x|^q$ is convex for all $q \geq 1$.

Now let $x,y \in W^{1,q}_{w}(t_1,t_2;\R^d)$. Then $\dot{x}(t),\dot{y}(t)\in \R^d$ for all $t\in(t_1,t_2)$, and so we may use the previous inequality to write
$$|\lambda \dot{x}(t) + (1-\lambda)\dot{y}(t)|^q\leq \lambda|\dot{x}(t)|^q+(1-\lambda)|\dot{y}(t)|^q$$ for all $\lambda \in [0,1]$. Since $w(t)>0$ for all $t \in [t_1,t_2]$, we multiply through by $\frac{1}{q} w(t)$ and integrate from $t_1$ to $t_2$ to obtain
\begin{align*}
    \frac{1}{q} \int_{t_1}^{t_2}|\lambda \dot{x}(t) + (1-\lambda)\dot{y}(t)|^q w(t) \ \mathrm{d}t&\leq \lambda\frac{1}{q}\int_{t_1}^{t_2}|\dot{x}(t)|^q w(t)\ \mathrm{d}t\\ &\hspace{1cm}+(1-\lambda)\frac{1}{q}\int_{t_1}^{t_2}|\dot{y}(t)|^q w(t)\ \mathrm{d}t.
\end{align*}
Hence, we conclude that $E$ is convex on $W^{1,q}_{w}(t_1,t_2;\R^d)$.
\end{proof}

In order to assist us in finding a minimiser of $E$, we introduce the notion of the G\^ateaux derivative, which can be understood as a generalisation of the classical directional derivative, which is sensible for functionals defined on (infinite-dimensional) Banach spaces. For more details regarding the application of the G\^ateaux derivatives to minimisation problems, we direct the reader to Appendix \ref{sec:convexmin}.

\begin{definition}[G\^ateaux derivative, \cite{Zeidler}]
Let $E:V \rightarrow (-\infty, +\infty]$ be a functional on a Banach space $V$. Then the \emph{G\^ateaux derivative} of $E$ at $x \in V$ in the direction $h \in V$ is given by
$$\lim_{t \rightarrow 0^+} \frac{E(x+th) - E(x)}{t}.$$
If this limit exists in all directions $h \in V$ and there exists $E'(x) \in V'$ such that
$$\langle E'(x), h\rangle_{V',V} = \lim_{t \rightarrow 0^+} \frac{E(x+th) - E(x)}{t}$$
for all $h \in V$, then we call $E$ G\^ateaux differentiable at $x$ with G\^ateaux derivative $E'(x)$.\end{definition}
We note that if the G\^ateaux derivative $E'(x)$ exists, then it is unique. In addition, a suitable candidate for the G\^ateaux derivative of $E$ at $x$ may be calculated as $$\left.\frac{d}{dt} E(x+th)\right|_{t=0}.$$ However, it is necessary to check that this function satisfies the definition in a rigorous sense, which we demonstrate below for the functional $E$ defined in \eqref{eq:general functional}.

\begin{proposition}\label{E is diff}
For a given function $\xi_0 \in W^{1,q}_{w}(t_1,t_2;\R^d)$, the mapping $x \mapsto E(\xi_0 + x)$ is G\^ateaux differentiable at all $x \in W^{1,q}_{w,0}(t_1,t_2;\R^d)$ with the G\^ateaux derivative at $x$ in the direction $h \in W^{1,q}_{w,0}(t_1,t_2;\R^d)$ given by
\begin{equation*}\label{eq:gatdiffe}
    \langle E'(\xi_0 + x), h\rangle_{W^{-1,q'}_0,W^{1,q}_0} := \int_{t_1}^{t_2} |\dot{\xi_0}(t) + \dot{x}(t)|^{q-2}\big(\dot{\xi_0}(t) + \dot{x}(t)\big)\dot{h}(t) w(t) \ \mathrm{d}t.
\end{equation*}
\end{proposition}

\begin{proof}
Fix $x \in W^{1,q}_{w,0}(t_1,t_2;\R^d)$ and set $v:=\xi_0 + x$. First, we verify that $E'(v)$ is indeed an element of the dual space $$(W^{1,q}_{w,0}(t_1,t_2;\R^d))'=: W^{-1,q'}_{w,0}(t_1,t_2;\R^d).$$ In particular, we show that $E'(v): W^{1,q}_{w,0}(t_1,t_2;\R^d) \rightarrow \R$ is linear and continuous. Let $h_1, h_2 \in W^{1,q}_{w,0}(t_1,t_2;\R^d)$ and $\lambda_1, \lambda_2 \in \R$ be arbitrary. Then, using the linearity of weak differentiation and integration,
\begin{align*}
    & \hspace{-0.5cm} \langle E'(v), \lambda_1 h_1 + \lambda_2 h_2\rangle_{W^{-1,q'}_{w,0},W^{1,q}_{w,0}} \\
    &= \int_{t_1}^{t_2} |\dot{v}(t)|^{q-2}\dot{v}(t)\big(\lambda_1\dot{h_1}(t) + \lambda_2\dot{h_2}(t)\big)w(t)\ \mathrm{d}t\\
    & = \lambda_1 \int_{t_1}^{t_2} |\dot{v}(t)|^{q-2}\dot{v}(t)\dot{h_1}(t) w(t)\ \mathrm{d}t + \lambda_2 \int_{t_1}^{t_2} |\dot{v}(t)|^{q-2}\dot{v}(t)\dot{h_2}(t)w(t)\ \mathrm{d}t\\
    & = \lambda_1 \langle E'(v),h_1\rangle_{W^{-1,q'}_{w,0},W^{1,q}_{w,0}}+\lambda_2 \langle E'(v),h_2\rangle_{W^{-1,q'}_{w,0},W^{1,q}_{w,0}}
\end{align*}
Therefore, $E'(v)$ is linear. Then, 
\begin{align*}
    & \hspace{-0.5cm} \big|\langle E'(v),h_1\rangle_{W^{-1,q'}_{w,0},W^{1,q}_{w,0}} - \langle E'(v),h_2\rangle_{W^{-1,q'}_{w,0},W^{1,q}_{w,0}}\big|\\
    &\leq  \int_{t_1}^{t_2} \big||\dot{v}(t)|^{q-2}\dot{v}(t)\big|\big|\dot{h_1}(t) - \dot{h_2}(t)\big|w(t) \ \mathrm{d}t\\
    &=\int_{t_1}^{t_2} |\dot{v}(t)|^{q-1}|\dot{h_1}(t) - \dot{h_2}(t)|w(t) \ \mathrm{d}t\\
    &\leq \left(\int_{t_1}^{t_2}|\dot{v}(t)|^{(q-1)q'} w(t) \ \mathrm{d}t\right)^{1/q'}\left(\int_{t_1}^{t_2}|\dot{h_1}(t) - \dot{h_2}(t)|^q w(t) \ \mathrm{d}t\right)^{1/q}
\end{align*}
by H\"older's inequality. The quantity
$$\int_{t_1}^{t_2}|\dot{v}(t)|^{(q-1)q'} w(t) \ \mathrm{d}t = \int_{t_1}^{t_2}|\dot{v}(t)|^{q} w(t) \ \mathrm{d}t$$
is finite since $|\dot{v}(t)| \in L^q_w(t_1,t_2;\R^d)$. Hence,
$$D:= \left(\int_{t_1}^{t_2}|\dot{v}(t)|^{q} w(t) \ \mathrm{d}t\right)^{1/q'}= \|\dot{v} \|_{q,w}^{q-1}<\infty.$$
It follows that 
\begin{align*}
    \big|\langle E'(v),h_1\rangle_{W^{-1,q'}_{w,0},W^{1,q}_{w,0}} - \langle E'(v),h_2\rangle_{W^{-1,q'}_{w,0},W^{1,q}_{w,0}}\big| &\leq D\|\dot{h_1} - \dot{h_2}\|_{q,w}\\
    &\leq D\|h_1 - h_2\|_{W^{1,q}_w}
\end{align*}
and therefore $E'(v)$ is also continuous. Thus, by the definition of the dual of a normed vector space, the linearity and continuity of $E'(v)$ imply that $E'(v) \in W^{-1,q'}_{w,0}(t_1,t_2;\R^d)$.

We now show $E'(v)$ is indeed the G\^ateaux derivative of $E$ at $v$, that is
$$\langle E'(v), h\rangle_{W^{-1,q'}_{w,0},W^{1,q}_{w,0}} = \lim_{s \rightarrow 0^+} \frac{E(v+sh)-E(v)}{s}$$ for all $h \in W^{1,q}_{w,0}(t_1,t_2;\R^d)$,
which is equivalent to proving
\begin{equation}\label{eq:gatlim}
    \lim_{s\rightarrow0^+} \left| \int_{t_1}^{t_2} \left( \frac{|\dot{v}(t) + s\dot{h}(t)|^q - |\dot{v}(t)|^q}{s} -|\dot{v}(t)|^{q-2}\dot{v}(t)\dot{h}(t)\right) w(t) \ \mathrm{d}t \right|=0.
\end{equation}
By applying the mean value theorem to the continuously differentiable map $v \mapsto |v|^q$ for every $t \in (t_1,t_2)$, there exists $\theta(t) \in (0,1)$ such that 
$$|\dot{v}(t) + s\dot{h}(t)|^q - |\dot{v}(t)|^q = |\dot{v}(t) + s\theta(t)\dot{h}(t)|^{q-2}\big(\dot{v}(t) + s\theta(t)\dot{h}(t)\big)s\dot{h}(t).$$
Inserting this into the limit in \eqref{eq:gatlim}, we calculate 
\begin{align*}
\begin{split}
    &\lim_{s\rightarrow 0^+} \bigg|\int_{t_1}^{t_2} \Big(|\dot{v}(t) + s\theta(t)\dot{h}(t)|^{q-2}\big(\dot{v}(t) + s\theta(t)\dot{h}(t)\big)\dot{h}(t) \\&\hspace{5.2cm}- |\dot{v}(t)|^{q-2}\dot{v}(t)\dot{h}(t)\Big) w(t) \ \mathrm{d}t \bigg|.
\end{split}
\end{align*}
Using the continuity of $v \mapsto |v|^{q-2}v$, 
$$\lim_{s \rightarrow 0^+} |\dot{v}(t) + s\theta(t)\dot{h}(t)|^{q-2}\big(\dot{v}(t) + s\theta(t)\dot{h}(t)\big)\dot{h}(t) = |\dot{v}(t)|^{q-2}\dot{v}(t)\dot{h}(t)$$
for all $t \in (t_1,t_2)$. Then, 
\begin{align*}
    &\hspace{-1cm} \big||\dot{v}(t) + s\theta(t)\dot{h}(t)|^{q-2}(\dot{v}(t) + s\theta(t)\dot{h}(t))\dot{h}(t)\big|\\
    & = |\dot{v}(t) + s\theta(t)\dot{h}(t)|^{q-1}|\dot{h}(t)|\\
    &\leq (|\dot{v}(t)| + s\theta(t)|\dot{h}(t)|)^{q-1}|\dot{h}(t)|
\end{align*}
Then, since $s\theta(t)>0$ and thus $s\theta(t)|\dot{h}(t)| \geq 0$, by an elementary inequality, there exists $K>1$ such that 
$$(|\dot{v}(t)| + s\theta(t)|\dot{h}(t)|)^{q-1} \leq K\big(|\dot{v}(t)|^{q-1} +(s\theta(t)|\dot{h}(t)|)^{q-1}\big)$$
for all $t \in (t_1,t_2)$. Therefore,
\begin{align*}
    &\hspace{-1cm} \big||\dot{v}(t) + s\theta(t)\dot{h}(t)|^{q-2}(\dot{v}(t) + s\theta(t)\dot{h}(t))\dot{h}(t)\big|\\
    & \leq K\big(|\dot{v}(t)|^{q-1} +(s\theta(t)|\dot{h}(t)|)^{q-1}\big)|\dot{h}(t)|\\
    & \leq K(|\dot{v}(t)|^{q-1} +|\dot{h}(t)|^{q-1})|\dot{h}(t)|=:g(t)
\end{align*}
since $s\theta(t) \in (0,1)$ for $s$ small enough. Since $\dot{v}(t),\dot{h}(t) \in L^q_w(t_1,t_2; \R^d)$ and are hence also in $L^{q-1}_w(t_1,t_2;\R^d)$, we have that $$g(t) = K(|\dot{v}(t)|^{q-1} +|\dot{h}(t)|^{q-1})|\dot{h}(t)| \in L^1_w(t_1,t_2;\R^d).$$ Then, by Lebesgue's dominated convergence theorem, 
\begin{equation*}
\begin{split}\lim_{s\rightarrow 0^+} \int_{t_1}^{t_2} \bigg|\bigg(|\dot{v}(t) + s\theta(t)\dot{h}(t)|^{q-2}\big(\dot{v}(t) + s\theta(t)\dot{h}(t)\big)\dot{h}(t) \\- |\dot{v}(t)|^{q-2}\dot{v}(t)\dot{h}(t)\bigg)\bigg| w(t) \ \mathrm{d}t =0.
\end{split}
\end{equation*}
Thus,
\begin{align*}
    0 &\leq \lim_{s\rightarrow0^+} \left| \int_{t_1}^{t_2} \left( \frac{|\dot{v}(t) + s\dot{h}(t)|^q - |\dot{v}(t)|^q}{s} -|\dot{v}(t)|^{q-2}\dot{v}(t)\dot{h}(t)\right) w(t) \ \mathrm{d}t \right|\\
    &\leq \lim_{s\rightarrow 0^+} \int_{t_1}^{t_2} \bigg|\bigg(|\dot{v}(t) + s\theta(t)\dot{h}(t)|^{q-2}\big(\dot{v}(t) + s\theta(t)\dot{h}(t)\big)\dot{h}(t)\\
    &\hspace{6.5cm}- |\dot{v}(t)|^{q-2}\dot{v}(t)\dot{h}(t)\bigg)\bigg| w(t) \ \mathrm{d}t \\
    &=0,
\end{align*}
which proves \eqref{eq:gatlim} and therefore, $E(v) = E(\xi_0 + x)$ is G\^ateaux differentiable with derivative $E'(v) = E'(\xi_0 + x)$ defined in Proposition \ref{E is diff}.
\end{proof}

We now ready to prove the main result of this section, Theorem \ref{generalised weight min}.

\begin{proof}
Since $W^{1,q}_{w,0}(t_1,t_2;\R^d)$ is a reflexive Banach space, we aim to apply Theorem \ref{convexminthm} from the Appendix, which provides sufficient conditions for the mapping $x \mapsto E(\xi_0 + x)$ to attain a minimum. Propositions \ref{E is coercive} and \ref{E is convex} provide that that this mapping is coercive, convex, and continuous on $W^{1,q}_{w,0}(t_1,t_2;\R^d)$, and thus also lower semicontinuous. Therefore, in view of Theorem \ref{convexminthm}, the mapping $x \mapsto E(\xi_0 + x)$ attains a minimum on $W^{1,q}_{w,0}(t_1,t_2;\R^d)$. Equivalently, letting $v=\xi_0 + x$, $E(v)$ attains a minimum on $\mathcal{A}$.

It is known from Proposition \ref{E is diff} that the mapping $x \mapsto E(\xi_0 +x)$ is G\^ateaux differentiable for all $x \in W^{1,q}_{w,0}(t_1,t_2;\R^d)$. Since $E$ is also convex, it follows from Propositions \ref{subdiffmin} and \ref{subdiffgateaux} that any point $v = \xi_0 + x \in \mathcal{A}$ which minimises $E(v) = E(\xi_0 + x)$ can be found as a point which satisfies
\begin{equation}\label{eq:predirichlet}
   \langle E'(v), h\rangle_{W^{-1,q'}_{w,0},W^{1,q}_{w,0}} = \int_{t_1}^{t_2} | \dot{v}(t)|^{q-2}\dot{v}(t)\dot{h}(t) w(t) \ \mathrm{d}t = 0
\end{equation}
for all $h \in W^{1,q}_{w,0}(t_1,t_2;\R^d)$.

Now, using that $\int_{t_1}^{t_2}|\dot{v}(t)|^q w(t) \ \mathrm{d}t < \infty$, we have that
$$\int_{t_1}^{t_2} \big||\dot{v}(t)|^{q-2}\dot{v}(t) w(t)\big|^{q'} \ \mathrm{d}t = \int_{t_1}^{t_2} |\dot{v}(t)|^{q} (w(t))^{q'} \ \mathrm{d}t< \infty,$$
therefore, $|\dot{v}(t)|^{q-2}\dot{v}(t) w(t) \in L^{q'}(t_1,t_2; \R^d)$. Then, write \eqref{eq:predirichlet} as 
$$\int_{t_1}^{t_2} |\dot{v}(t)|^{q-2}\dot{v}(t)\dot{h}(t) w(t) \ \mathrm{d}t = 0 = -\int_{t_1}^{t_2} 0 \cdot h(t) \ \mathrm{d}t$$
for all $h \in W^{1,q'}_{w,0}(t_1,t_2;\R^d)$ and thus for all $h \in C_c^\infty(t_1,t_2;\R^d)$. Using the definition of the Sobolev space $W^{1,q'}(t_1,t_2;\R^d)$ and the uniqueness of the weak derivative, in order for \eqref{eq:predirichlet} to hold, one must have that
$$\{|\dot{v}(t)|^{q-2}\dot{v}(t)w(t)\}'=0$$
for a.e. $t \in (t_1, t_2)$. Thus, the solution to the minimisation problem can be found as the solution to the Dirichlet problem
\begin{equation}\label{eq:thedirichlet}
\begin{cases}
\{|\dot{v}(t)|^{q-2}\dot{v}(t)w(t)\}'=0 & \text{a.e. on $(t_1, t_2)$},\\
v(t_1)=x_1, v(t_2)=x_2.
\end{cases}
\end{equation}
Here, the boundary conditions ensure the solution is an element of $\mathcal{A}$.

The equation $\{|\dot{v}(t)|^{q-2}\dot{v}(t)w(t)\}'=0$ is satisfied if and only if $$|\dot{v}(t)|^{q-2}\dot{v}(t)w(t) = A$$ for some $A \in \R^d$. Then, since $t$ is nonzero, the equation may be written as
\begin{equation}\label{eq:dirstep}
    |\dot{v}(t)|^{q-2}\dot{v}(t) = A(w(t))^{-1}.
\end{equation}
Consider the function $k(v) = |v|^{q-2}v$ for $v \in \R^d$ and let $k^{-1}(v) := |v|^{q'-2}v$. We see that $k^{-1}$ is the inverse of $k$ by checking $(k^{-1} \circ k)(v) = v$ for all $v \in \R^d$. We have that 
\begin{align*}
    (k^{-1} \circ k)(v) &= \big| |v|^{q-2}v\big|^{q'-2}|v|^{q-2}v\\
    &=|v|^{(q-2)(q'-2) + q + q' -4}v\\
    &=v
\end{align*}
since $(q-2)(q'-2) + q + q' -4=0$. Applying the function $k^{-1}$ to both sides of \eqref{eq:dirstep}, we have that
\begin{align*}
    \dot{v}(t) &= |A(w(t))^{-1}|^{q'-2}A\big(w(t)\big)^{-1}\\
    &= |A|^{q'-2}A \big(w(t)\big)^{1-q'}\\
    &=B\big(w(t)\big)^{\frac{1}{1-q}}
\end{align*}
where $B:=|A|^{q'-2}A$ is an element of $\R^d$. Let $W(t)$ denote an antiderivative of $\big(w(t)\big)^{\frac{1}{1-q}}$. Then the solution to the Dirichlet problem \eqref{eq:thedirichlet} is given by
\begin{equation}\label{eq:dirichletsol}
    v(t) = D + BW(t),
\end{equation}
where $B,D \in \R^d$ are given by
$$B = \frac{x_2 - x_1}{W(t_2) - W(t_1)}, \quad D = x_1 - BW(t_1).$$ We note that since $w(t)$ is continuous on $[t_1,t_2]$, $v(t)$ is continuously differentiable. Thus, $v(t)$ is a $C^1$ path connecting the points $(x_1,t_1)$ and $(x_2,t_2)$ in $\R^d \times (0,\infty)$, which will become relevant later.

We prove this solution is unique, which is equivalent to proving that the minimiser of $E$ is unique. According to Theorem \ref{convexminthm}, one may also arrive at this conclusion by proving the functional $E$ is strictly convex. However, we prove the uniqueness of solutions to the Dirichlet problem \eqref{eq:thedirichlet} as an alternate justification for this result.

Suppose $v, \tilde{v} \in \mathcal{A}$ are two solutions of \eqref{eq:thedirichlet} and let $y:=v-\tilde{v}$. Since $v$ and $\tilde{v}$ both satisfy the boundary conditions of this problem, it follows $y(t_1)=y(t_2)=0$ and therefore $y \in W^{1,q}_{w,0}(t_1,t_2;\R^d)$. Recall from \eqref{eq:predirichlet} that
$$\int_{t_1}^{t_2} |\dot{v}(t)|^{q-2}\dot{v}(t)\dot{h}(t) w(t) \ \mathrm{d}t = 0$$ for all $h \in W^{1,q}_{w,0}(t_1,t_2;\R^d)$. A similar statement can be obtained by replacing $v$ by $\tilde{v}$. Applying this for $h=y= v-\tilde{v}$, subtracting and factorising, we have that
\begin{equation}\label{eq:inttilde}
    \int_{t_1}^{t_2} \big(|\dot{v}(t)|^{q-2}\dot{v}(t) - |\dot{\tilde{v}}(t)|^{q-2}\dot{\tilde{v}}(t)\big)\big(\dot{v}(t) -\dot{\tilde{v}}(t)\big) w(t) \ \mathrm{d}t = 0.
\end{equation}
As shown in the proof of Proposition \ref{E is convex}, the function $f(x):=|x|^q$ is convex, and therefore its derivative $f'(x)=|x|^{q-2}x$ is a monotone function. Hence, $f'(x)$ satisfies the following monotonicity condition, which we write here in terms of $\dot{v}(t)$ and $\dot{\tilde{v}}(t)$ for all $t \in (t_1,t_2)$:
\begin{equation}\label{eq:mono}
    \big(|\dot{v}(t)|^{q-2}\dot{v}(t) - |\dot{\tilde{v}}(t)|^{q-2}\dot{\tilde{v}}(t)\big)\big(\dot{v}(t) -\dot{\tilde{v}}(t)\big) \geq 0
\end{equation}
Together, \eqref{eq:inttilde} and \eqref{eq:mono} imply that
$$\big(|\dot{v}(t)|^{q-2}\dot{v}(t) - |\dot{\tilde{v}}(t)|^{q-2}\dot{\tilde{v}}(t)\big)\big(\dot{v}(t) -\dot{\tilde{v}}(t)\big) =0 \quad \text{a.e. } t\in (t_1,t_2).$$
From this, we conclude that $\dot{v}(t) = \dot{\tilde{v}}(t)$ for a.e. $t \in (t_1,t_2)$ and therefore, $v(t) = \tilde{v}(t) +C$ for some constant $C \in \R^d$. Since $v$ and $\tilde{v}$ both satisfy the boundary conditions given in \eqref{eq:thedirichlet}, we must have that $C=0$. Therefore, $v(t)=\tilde{v}(t)$ for a.e. $t \in (t_1,t_2)$ and the solution to the problem given by \eqref{eq:dirichletsol} is unique. Thus, it follows that the point $v \in \mathcal{A}$ for which $E$ attains its minimum is unique. 

Finally, we use the unique solution $v$ of \eqref{eq:thedirichlet} to calculate that the minimum value of the functional $E$ on $\mathcal{A}$ is given by
\begin{align*}
    E(v) &= \frac{1}{q}\int_{t_1}^{t_2} |\dot{v}(t)|^q w(t) \ \mathrm{d}t\\
    &=\frac{|B|^q}{q} \int_{t_1}^{t_2} \big(w(t)\big)^{\frac{q}{1-q}} w(t) \ \mathrm{d}t\\
    &=\frac{|B|^q}{q}\int_{t_1}^{t_2} \big(w(t)\big)^{\frac{1}{1-q}} \ \mathrm{d}t\\
    &=\frac{|B|^q}{q}\big(W(t_2) - W(t_1) \big)
\end{align*}
This is equivalent to the result, which was claimed in Theorem \ref{generalised weight min}.
\end{proof}

Lastly, we consider a special case of Theorem \ref{generalised weight min}, corresponding to the unweighted functional $E(x):=\frac{1}{q}\int_{t_1}^{t_2} |\dot{x}(t)|^q \ \mathrm{d}t$. The following corollary may be deduced from Theorem \ref{generalised weight min} by setting $w(t)\equiv 1$.

\begin{corollary}\label{cor:unweighted}
Let $1<q<\infty$, $0<t_1<t_2$, and $x_1, x_2 \in \R^d$ and \mbox{$\mathcal{A}:=\{\xi_0\} \oplus W^{1,q}_{0}(t_1,t_2;\R^d)$}. Let $E:W^{1,q}(t_1,t_2;\R^d) \rightarrow \R$ be the functional defined by 
$$E(x):=\frac{1}{q}\int_{t_1}^{t_2} |\dot{x}(t)|^q \ \mathrm{d}t.$$ Then,
$$\min_{x \in \mathcal{A}} E(x) = \frac{|x_2-x_1|^q}{q (t_2 - t_1)^{q-1}}.$$
\end{corollary}

\chapter{General Harnack-type Inequalities}\label{sec:genharnack}

In this chapter, we prove a collection of Harnack inequality results discovered by Auchmuty and Bao \cite{Auchmuty-Bao-1994} for functions satisfying a gradient estimate of a particular form. Throughout, we let $\Omega \subseteq \R^d$ be an open convex set and write $\Omega_T:=\Omega \times (0,T)$.
\begin{proof}[Proof of Theorem \ref{thm:generalharnackineq}]
Let $f:\Omega_T \rightarrow (0,\infty)$ be continuously differentiable. Suppose there are  $C>0$, $p>1$, $r \in \R$ and $a \in L_{\loc}^1(0,T)$  such that $f$ satisfies the inequality
\begin{equation*}\tag{\ref{eq:thm1}}
    \frac{\partial f}{\partial t} + af \geq \frac{C|\nabla f|^p}{f^r} \qquad \text{in } \Omega_T.
\end{equation*}
We first aim to manipulate inequality \eqref{eq:thm1} in order to apply Young's inequality in the form
\begin{equation}\label{eq:young}
    \frac{1}{p}|Y|^p + \frac{1}{q}|Z|^q \geq Y \cdot Z
\end{equation}
for some vectors $Y, Z \in \R^d$. To do this, let $\phi: \Omega_T \rightarrow \R$ and $W: \Omega \rightarrow \R^d$ be continuous functions. By adding the quantity $-(\nabla f)\cdot W + \frac{1}{q}|\phi W|^q$ to both sides of \eqref{eq:thm1}, we have
\begin{equation}\label{eq:calculation1}
    \frac{\partial f}{\partial t} + af -(\nabla f)\cdot W + \frac{1}{q}|\phi W|^q\geq \frac{C|\nabla f|^p}{f^r}-(\nabla f)\cdot W + \frac{1}{q}|\phi W|^q.
\end{equation}
If one chooses $\phi = (\frac{f^r}{pC})^{\frac{1}{p}}$, then the right-hand side of \eqref{eq:calculation1} satisfies
\begin{align*}
    & \frac{C|\nabla f|^p}{f^r}-(\nabla f)\cdot W + \frac{1}{q}\left|\left(\frac{f^r}{pC}\right)^{\frac{1}{p}} W\right|^q\\
    &\hspace{2cm}=\frac{pC}{f^r} \left[\frac{1}{p}|\nabla f|^p-(\nabla f)\cdot \left(\frac{f^r}{pC}W\right) + \frac{1}{q}\left|\left(\frac{f^r}{pC}\right) W\right|^q\right],
\end{align*}
which is nonnegative due to Young's inequality \eqref{eq:young} applied to $Y=\nabla f$ and $Z=\left(\frac{f^r}{pC}\right)W$. Combining this with \eqref{eq:calculation1} then gives
$$\frac{\partial f}{\partial t} + af -(\nabla f)\cdot W + \frac{1}{q}\left|\left(\frac{f^r}{pC}\right)^{\frac{1}{p}} W\right|^q \geq 0.$$
By letting $\xi:=\frac{1}{q}\left(\frac{1}{pC}\right)^{q-1}$, we can rewrite the last inequality as 
\begin{equation}\label{eq:ineq}
    \frac{\partial f}{\partial t} + af -(\nabla f)\cdot W + \xi f^{\frac{r}{p-1}} \left|W\right|^q \geq 0.
\end{equation}

Now, let $(x_1,t_1), (x_2,t_2) \in \Omega_T$ be arbitrary for some $0<t_1<t_2<T$, and let $x(t)$ be a $C^1$ path in $\Omega$ such that $x(t_1) = x_1$ and $x(t_2) = x_2$, that is, the path $x(t)$ connects the points $x_1$ and $x_2$. By the chain rule, it follows that $$ \frac{\partial f}{\partial t}(x(t),t)=\frac{d}{dt} f(x(t),t) - (\nabla f(x(t),t))\cdot (\dot{x}(t)) $$ for all $t \in (t_1,t_2)$. Inserting this into \eqref{eq:ineq} yields that along the path $x(t)$, $f$ satisfies $$\frac{df}{dt} + a f - (\nabla f)\cdot \dot{x} - (\nabla f)\cdot W + \xi f^{\frac{r}{p-1}}|W|^q \geq 0.$$ Thus, we choose $W=-\dot{x}$ and set $m+1=\frac{r}{p-1}$ to simplify this inequality as 
\begin{equation}\label{eq:ineq2}
  \frac{df}{dt} + af \geq -\xi f^{m+1}|\dot{x}|^q.  
\end{equation}

We divide the remainder of the proof into three cases: (i) $m=0$, (ii) $m > 0$, and (iii) $m<0$. Throughout, $A(t):=\int^t a(s) \ ds$ will denote an antiderivative of $a \in L^1_{\loc}(0,T)$.

We begin by considering the case (i) $m = 0$. Then, \eqref{eq:ineq2} becomes $$\frac{df}{dt} + af \geq -\xi f |\dot{x}|^q.$$ Since $f$ is positive, we can rewrite this inequality as 
\begin{equation}\label{eq:calculation2}
    \frac{1}{f} \frac{df}{dt} + a \geq -\xi |\dot{x}|^q.
\end{equation}
Integrating \eqref{eq:calculation2} over $(t_1,t_2)$ yields that 
\begin{equation}\label{eq:ineq3}
    \log[f(x_2,t_2)] + A(t_2) + \xi\int_{t_1}^{t_2} |\dot{x}(t)|^q \ \mathrm{d}t \geq \log[f(x_1,t_2)] + A(t_1).
\end{equation}

Next, we optimise the estimate \eqref{eq:ineq3} among all curves connecting $x_1$ at $t_1$ and $x_2$ at $t_2$. According to Corollary \ref{cor:unweighted},  the minimum value of the integral $\int_{t_1}^{t_2} |\dot{x}(t)|^q \ \mathrm{d}t$ is given by $$\int_{t_1}^{t_2} \left|\frac{x_2-x_1}{t_2-t_1} \right|^q \ \mathrm{d}t = \frac{|x_2 - x_1|^q}{(t_2-t_1)^{q-1}}.$$
Applying this to \eqref{eq:ineq3}, we get that 
$$\log[f(x_2,t_2)] + A(t_2) + \xi\frac{|x_2 - x_1|^q}{(t_2-t_1)^{q-1}} \geq \log[f(x_1,t_2)] + A(t_1).$$ By exponentiating and rearranging the last inequality, we obtain 
$$f(x_2,t_2)\geq \frac{e^{A(t_1)}}{e^{A(t_2)}}f(x_1,t_1)e^{-\frac{\xi|x_2-x_1|^q}{(t_2-t_1)^{q-1}}},$$
which is, in fact, result (i) in Theorem \ref{thm:generalharnackineq}.

(ii) $m > 0$: We multiply \eqref{eq:ineq2} by the integrating factor $e^{A}$ to obtain $$\frac{d}{dt} (fe^A) \geq -\xi f^{m+1}e^A|\dot{x}|.$$
Since $f$ is positive, this can be rewritten as $$\frac{1}{(fe^A)^{m+1}}\frac{d}{dt}(fe^A) \geq -\xi e^{-mA}|\dot{x}|^q.$$
Since $m > 0$, it follows that
\begin{equation*}
    \frac{-m}{(fe^A)^{m+1}}\frac{d}{dt}(fe^A) \leq m\xi e^{-mA}|\dot{x}|^q.
\end{equation*}
By integrating over $(t_1,t_2)$, we obtain 
\begin{equation}\label{eq:ineq4}
    [f(x_2,t_2)e^{A(t_2)}]^{-m} \leq [f(x_1,t_1)e^{A(t_1)}]^{-m} + m\xi\int_{t_1}^{t_2} e^{-mA}|\dot{x}(t)|^q \ \mathrm{d}t.
\end{equation}

Since $A(t)$ is continuous as the antiderivative of a locally integrable function, the function $w(t):=e^{-mA(t)}$ is continuous. As well, observe that $w(t)>0$ for all $t \in [t_1,t_2]$. Hence, by Theorem \ref{generalised weight min}, the functional $$E(x):=\frac{1}{q}\int_{t_1}^{t_2} e^{-mA(t)}|\dot{x}(t)|^q \ \mathrm{d}t$$ has a minimum value of
\begin{equation*}
    \frac{1}{q}|x_2-x_1|^q\left( \displaystyle{\int_{t_1}^{t_2}} e^{m(p-1)A(t)} \ \mathrm{d}t \right)^{1-q} = \frac{1}{q}|x_2-x_1|^qI^{1-q},
\end{equation*}
where we define $I:=\int_{t_1}^{t_2} e^{m(p-1)A(t)} \ \mathrm{d}t$. Inserting this information into \eqref{eq:ineq4}, we have that 
\begin{align*}
    [f(x_2,t_2)e^{A(t_2)}]^{-m} &\leq [f(x_1,t_1)e^{A(t_1)}]^{-m} + m\xi|x_2-x_1|^qI^{1-q} \\
    f(x_2,t_2)e^{A(t_2)} &\geq \left\{[f(x_1,t_1)e^{A(t_1)}]^{-m} + m\xi|x_2-x_1|^qI^{1-q}\right\}^{\frac{-1}{m}} \\
    \begin{split}f(x_2,t_2)&\geq\frac{e^{A(t_1)}}{e^{A(t_2)}}f(x_1,t_1)\times \\ &\qquad (1+m\xi|x_2-x_1|^qI^{1-q}[f(x_1,t_1)]^m e^{mA(t_1)})^{\frac{-1}{m}},\end{split}
\end{align*}
which proves the claim of Theorem \ref{thm:generalharnackineq} in case (ii).

(iii) $m < 0$: The working in this case follows similarly to that in case (ii). We obtain that 
$$\frac{1}{(fe^A)^{m+1}}\frac{d}{dt}(fe^A) \geq -\xi e^{-mA}|\dot{x}|^q.$$ Using that $m<0$ and multiplying both sides of the above inequality by $-m$ yields that 
\begin{equation*}
    \frac{-m}{(fe^A)^{m+1}}\frac{d}{dt}(fe^A) \geq m\xi e^{-mA}|\dot{x}|^q.
\end{equation*}
By integrating and rearranging, 
\begin{align}
[f(x_1,t_1)e^{A(t_1)}]^{|m|} &\leq [f(x_2,t_2)e^{A(t_2)}]^{|m|} + |m|\xi\int_{t_1}^{t_2} e^{-mA}|\dot{x}(t)|^q \ \mathrm{d}t \label{eq:ineq5}
\end{align}
The integral in  inequality \eqref{eq:ineq5} is already known to have minimum value $|x_2-x_1|^qI^{1-q}$. Inserting this into \eqref{eq:ineq5} gives
$$[f(x_1,t_1)e^{A(t_1)}]^{|m|} \leq [f(x_2,t_2)e^{A(t_2)}]^{|m|} + |m| \, \xi \,|x_2-x_1|^qI^{1-q} .$$
This can be rearranged to give 
$$[f(x_2,t_2)]^{|m|}\geq \left( \frac{e^{A(t_1)}}{e^{A(t_2)}} \right)^{|m|} \left( [f(x_1,t_1)]^{|m|} - \frac{|m|\, \xi \,|x_2-x_1|^qI^{1-q}}{e^{|m|A(t_1)}} \right).$$
Since $f$ is positive, this can be rewritten as 
$$[f(x_2,t_2)]^{|m|}\geq \left( \frac{e^{A(t_1)}}{e^{A(t_2)}} \right)^{|m|}[f(x_1,t_1)]^{|m|} \left( 1 - \frac{|m|\, \xi \,|x_2-x_1|^qI^{1-q}}{[f(x_1,t_1)]^{|m|}e^{|m|A(t_1)}} \right).$$ If the quantity in the large parentheses is nonnegative, then we can take the $|m|^{th}$ root of both sides to obtain
$$f(x_2,t_2)\geq \frac{e^{A(t_1)}}{e^{A(t_2)}}f(x_1,t_1) \left( 1 - \frac{|m|\, \xi \,|x_2-x_1|^qI^{1-q}}{[f(x_1,t_1)]^{|m|}e^{|m|A(t_1)}} \right)^{\frac{1}{|m|}},$$
which concludes the proof in the case that $f$ is positive.

Finally, if $f$ is nonnegative and satisfies \eqref{eq:thm1} with $r=0$, an analogous argument can be used to demonstrate that $f$ satisfies $$\frac{d}{dt}(fe^A) \geq -\xi e^{A}|\dot{x}|^q.$$ Then inequality \eqref{eq:caseiiia} follows as in case (iii) with $m$ replaced by $-1$, which completes the proof.\end{proof}

Although the function $f$ was assumed to be continuously differentiable in Theorem \ref{thm:generalharnackineq}, in several applications, $f$ may only satisfy a weaker set of regularity assumptions. For instance, $f$ may be a weak solution to a particular parabolic equation. Thus, we state and prove a second version of this theorem, which is also due to Auchmuty and Bao \cite{Auchmuty-Bao-1994}. In particular, we will place the following weaker assumptions on the function $f$:
\begin{enumerate}[label={(F\arabic*):}]
    \item $f\in W^{1,1}_{\loc}(\Omega_T)$ is positive and continuous on $\Omega_T$;
    \item $f$ is absolutely continuous on every continuously differentiable curve $x:[t_1,t_2] \rightarrow \Omega$ where $0 < t_1 < t_2 < T$;
    \item the intersection of the set of points at which $f$ does not have a classical total derivative and the image of any absolutely continuous curve in $\Omega$ is a set of $\mathcal{H}^1$ measure zero.
\end{enumerate}
These assumptions are chosen carefully, so that the calculus results used in the argument of the proof of Theorem \ref{thm:generalharnackineq} remain valid. For more detail regarding the underlying theory, we direct the reader to Appendix \ref{sec:abs-cont}.

\begin{theorem}[General Harnack inequalities II, \cite{Auchmuty-Bao-1994}]\label{thm:gen-dist}
Let $\Omega \subseteq \R^d$ be an open convex set. Suppose a function $f$ defined on $\Omega_T$ satisfies assumptions (F1)--(F3) and that
the inequality  
    \begin{equation*}\tag{\ref{eq:thm1}}
    \frac{\partial f}{\partial t} + af \geq \frac{C|\nabla f|^p}{f^r} \qquad \text{a.e. in } \Omega_T
    \end{equation*}
    holds for some constants $C>0$, $p>1$, $r \in \R$ and for some function $a \in L^1_{\loc}(0,T)$, where $\frac{\partial f}{\partial t} \in L^1_{\loc}(\Omega_T)$ and $\nabla f \in (L^1_{\loc}(\Omega_T))^d$ are to be understood as a weak partial derivative and a weak gradient respectively.
Then conclusions (i)--(iii) of Theorem \ref{thm:generalharnackineq} hold for all $x_1,x_2 \in \Omega$ and $0 < t_1 < t_2 <T$.

Moreover, if $f \in W^{1,1}_{\loc}(\Omega_T)$ is nonnegative and continuous on $\Omega_T$, satisfies assumptions (F2) and (F3), and inequality \eqref{eq:thm1} holds with $r=0$, then inequality \eqref{eq:caseiiia} holds as before for all $x_1,x_2 \in \Omega$ and $0 < t_1 < t_2 <T$.
\end{theorem}

\begin{proof}
The proof of this result closely follows the arguments presented in the proof of Theorem \ref{thm:generalharnackineq}, with the main change that the classical derivatives must now be replaced by weak derivatives. First assuming that $f$ is positive, we recall that the chain rule was previously used to write
\begin{equation}\label{eq:chain}
    \frac{\partial f}{\partial t}(x(t),t) = \frac{d}{dt}f(x(t),t) - (\nabla f(x(t),t))\cdot (\dot{x}(t))
\end{equation} for any continuously differentiable path $x:[t_1,t_2] \rightarrow \Omega$ connecting the points $(x_1,t_1), (x_2,t_2) \in \Omega_T$. Since $[t_1,t_2]$ is a compact subset of $(0,T)$ and by assumption $\frac{\partial f}{\partial t} \in L^1_{\loc}(0,T)$,  $\frac{\partial f}{\partial x_i} \in L^1_{\loc}(0,T)$ and $\frac{dx_i}{dt} \in C(t_1,t_2)$ for all $i=1,\ldots,d$, the quantity $(\nabla f) \cdot \dot{x} + \frac{\partial f}{\partial t}$ will satisfy $$\int_{t_1}^{t_2} \left| (\nabla f) \cdot \dot{x} + \frac{\partial f}{\partial t}\right| \ \mathrm{d}t < \infty.$$
Thus, in view of Theorem \ref{thm:ftc-abs}, this result, along with the assumptions (F2) and (F3), justifies the validity of the chain rule \eqref{eq:chain}. Therefore, we once again obtain  \begin{equation*}\tag{\ref{eq:ineq2}}\frac{df}{dt} + af \geq -\xi f^{m+1} |\dot{x}|^q\end{equation*} where $m+1 = \frac{r}{p-1}$.

As before, the proof can be divided into three cases (i) $m=0$, (ii) $m>0$, and (iii) $m<0$, which each involve integrating the inequality \eqref{eq:ineq2}. As an example, in case (i) we obtain $$\int_{t_1}^{t_2} \frac{d}{dt} (\log f(t)) \ \mathrm{d}t + A(t_2) + \xi\int_{t_1}^{t_2} |\dot{x}(t)|^q \ \mathrm{d}t \geq A(t_1).$$ In order to justify applying the fundamental theorem of calculus (Theorem \ref{thm:ftc-abs}) to the integral $\int_{t_1}^{t_2} \frac{d}{dt} (\log f(t)) \ \mathrm{d}t$, we must explain why the function $\log f $ is absolutely continuous on $[t_1,t_2]$. Indeed, the logarithm function is Lipschitz continuous on any interval $[\alpha,\beta]$, $0< \alpha <  \beta < \infty$ not containing zero. Since $f$ was assumed to be absolutely continuous on the image of any continuously differentiable curve, we conclude by Proposition \ref{prop:lip-abs} that $\log f$ is absolutely continuous on $[t_1,t_2]$ as the precomposition of an absolutely continuous function by a Lipschitz continuous function.

Similarly, during the proof in cases (ii) and (iii), we encounter the integral \begin{equation}\label{eq:integral}\int_{t_1}^{t_2} \frac{d}{dt} \Big((fe^{A(t)})^{-m} \Big) \ \mathrm{d}t\end{equation} for $m \ne 0$. The function $A(t)=\int^t a(s) \ \mathrm{d}s$ is absolutely continuous by Proposition \ref{prop:prim-abs} as the primitive of an integrable function. Since the exponential function is Lipschitz continuous on compact intervals, the composition $e^{A(t)}$ will be absolutely continuous on $[t_1,t_2]$. Then $fe^{A(t)}$ is absolutely continuous as the product of absolutely continuous functions. Finally, using that $fe^{A(t)}>0$ on $[t_1,t_2]$ and that the map $s \mapsto s^{-m}$ is Lipschitz continuous on any interval $[\alpha,\beta]$, $0<\alpha<\beta < \infty$ not containing zero, we conclude that $(fe^{A(t)})^{-m}$ is absolutely continuous on $[t_1,t_2]$. Thus, the fundamental theorem of calculus may be applied to the integral \eqref{eq:integral}.

Beyond this step, the remainder of the proof is unchanged, including the variational arguments used to improve the bound in the inequality.

Lastly, the result when $f$ is nonnegative and $r=0$ follows by the same reasoning. In particular, since $fe^{A(t)}$ is absolutely continuous, the fundamental theorem of calculus may be used to write $$\int_{t_1}^{t_2} \frac{d}{dt} (fe^{A(t)}) \ \mathrm{d}t = f(x(t_2),t_2)e^{A(t_2)} - f(x(t_1),t_1)e^{A(t_1)}.$$
\end{proof}

Finally, we state and prove a related result, which will be useful in applications.

\begin{theorem}[General Harnack inequalities III, \cite{Auchmuty-Bao-1994}]\label{thm:gen-backwards} Let $\Omega$ be an open convex set in $\R^d$. Suppose a function $f$ defined on $\Omega_T$ satisfies (F1)--(F3) and that the inequality \begin{equation}\label{eq:thm4}
    \frac{\partial f}{\partial t} + af \leq - \frac{C|\nabla f|^p}{f^r} \qquad \text{a.e. in } \Omega_T
\end{equation}
holds for some constants $C>0$, $p>1$, $r \in \R$ and for some function $a \in L^1_{\loc}(0,T)$, where $\frac{\partial f}{\partial t}$ and $\nabla f$ are interpreted in a weak sense. Then, for all $x_1,x_2 \in \Omega$ and $0 < t_1<t_2 < T$:
\begin{enumerate}
    \item if $r=p-1$, then
    \begin{equation*}\label{eq:thm4-casei}
    f(x_2,t_2)\leq \frac{e^{A(t_1)}}{e^{A(t_2)}}f(x_1,t_1)e^{\frac{\xi|x_2-x_1|^q}{(t_2-t_1)^{q-1}}};
    \end{equation*}
    \item if $r > p-1$, then
    \begin{equation}\label{eq:thm4-caseiia}
    \begin{split}
    [f(x_2,t_2)]^{-m}&\geq \frac{e^{mA(t_1)}}{e^{mA(t_2)}}\times \\&\qquad \Big(f(x_1,t_1)^{-m}-m\xi|x_2-x_1|^qI^{1-q} e^{mA(t_1)}\Big)^{\frac{-1}{m}}\end{split}
    \end{equation}
    and if the quantity inside the large parentheses is nonnegative, then
    \begin{equation*}\label{eq:thm4-caseiib}\begin{split}
    f(x_2,t_2)&\leq \frac{e^{A(t_1)}}{e^{A(t_2)}}f(x_1,t_1)\times \\&\qquad (1-m\xi|x_2-x_1|^qI^{1-q}[f(x_1,t_1)]^m e^{mA(t_1)})^{\frac{-1}{m}};\end{split}
    \end{equation*}
    \item if $r<p-1$, then
    \begin{equation*}\label{eq:thm4-caseiii}
    f(x_2,t_2)\leq \frac{e^{A(t_1)}}{e^{A(t_2)}}f(x_1,t_1) \left( 1 + \frac{|m|\, \xi \,|x_2-x_1|^qI^{1-q}}{[f(x_1,t_1)]^{|m|}e^{|m|A(t_1)}} \right)^{\frac{1}{|m|}},
    \end{equation*}
\end{enumerate}
where the quantities $q$, $m$, $\xi$, $A$, and $I$ are defined as in Theorem \ref{thm:generalharnackineq}.

Moreover, if $f$ is nonnegative on $\Omega_T$ and satisfies \eqref{eq:thm4} with $r=0$, then 
    \begin{equation*}\label{eq:thm4-nonneg}
    [f(x_2,t_2)]^{|m|}\leq \left( \frac{e^{A(t_1)}}{e^{A(t_2)}} \right)^{|m|} \left( [f(x_1,t_1)]^{|m|} + \frac{|m|\, \xi \,|x_2-x_1|^qI^{1-q}}{e^{|m|A(t_1)}} \right)
    \end{equation*}
for all $x_1,x_2 \in \Omega$ and $0 < t_1<t_2 < T$.
\end{theorem}

\begin{proof}
The proof once again follows from the arguments of Theorems \ref{thm:generalharnackineq} and \ref{thm:gen-dist}, so we summarise the main changes. We first add the quantity $(\nabla f) \cdot W - \tfrac{1}{q}|\phi W|^q$ to both sides of the inequality \eqref{eq:thm4}, where as before, $\phi = \left (\frac{f^r}{pC}\right)^{1/p}$ and $W$ is a continuous vector field on $\Omega$. By a similar argument to that used to prove Theorem \ref{thm:generalharnackineq}, the right-hand side of the resulting inequality will be nonpositive a.e. in $\Omega_T$. Thus, we obtain \begin{equation}\label{eq:calculation3}\frac{\partial f}{\partial t} + af + (\nabla f)\cdot W - \xi f^{\frac{r}{p-1}}|W|^q \leq 0.\end{equation}

As per the previous proofs, let $x$ be a continuously differentiable curve in $\Omega$ defined on $[t_1,t_2]$ with $0<t_1<t_2 < T$. By choosing $W=\dot{x}$ in \eqref{eq:calculation3} we reach $$\frac{\partial f}{\partial t} + (\nabla f) \cdot \dot{x} + af \leq \xi f^{m+1}|\dot{x}|^q,$$ where again we have let $m+1 = \frac{r}{p-1}$. As justified in the proof of Theorem \ref{thm:gen-dist}, we may apply the chain rule to obtain 
$$\frac{df}{dt} + af \leq \xi f^{m+1}|\dot{x}|^q.$$

The remainder of the proof follows the arguments of Theorem \ref{thm:generalharnackineq} with the inequality signs reversed and $-\xi$ replaced by $\xi$. In addition, the fundamental theorem of calculus may be applied as in the proof of Theorem \ref{thm:gen-dist}. As well, the variational arguments used to improve the bounds in the integrals still apply.

We also note one difference in the final form of Theorem \ref{thm:gen-backwards} compared with Theorems \ref{thm:generalharnackineq} and \ref{thm:gen-dist}. We see that in the case (ii) $m>0$, we may now only take the $-m^{\text{th}}$ root of both sides of inequality \eqref{eq:thm4-caseiia} if the quantity in large parentheses is nonnegative. This was not an issue in the corresponding result in Theorems \ref{thm:generalharnackineq} and \ref{thm:gen-dist}, since here, both terms in the parentheses were guaranteed to be positive. However, this does not remain true after exchanging $-\xi$ and $\xi$.

Finally, a similar argument applies to prove the result in the case when $f$ is nonnegative and $r=0$.

\end{proof}

\chapter{Applications to Nonlinear Evolution Equations}\label{sec:applications}

In this chapter, we apply the theorems from Chapter \ref{sec:genharnack} to obtain Harnack inequality results pertaining to the three significant examples of parabolic equations introduced in Chapter \ref{sec:intro}. Throughout, we work on the domain $\R^d_T:=\R^d \times (0,T)$.

\section{The Heat Equation}

We now demonstrate the calculation used in the proof of Theorem \ref{thm:generalharnackineq} to verify the Harnack inequality \eqref{eq:heatineq} derived by Li and Yau \cite{LiYau1986} for positive solutions $u$ to the heat equation  $$u_t = \Delta u \qquad \text{ in } \R^d_T.$$ In \cite{LiYau1986} the gradient estimate 
\begin{equation*}\tag{\ref{eq:liyauestimate}}\frac{|\nabla u|^2}{u^2} - \frac{u_t}{u} \leq \frac{d}{2t} \qquad \text{in } \R^d_T\end{equation*} was derived. By rearranging this inequality, we have that
\begin{equation}\label{eq:estimaterearranged}
    \frac{\partial u}{\partial t} + \frac{d}{2t}u \geq \frac{|\nabla u|^2}{u},
\end{equation}
which is of the form \eqref{eq:thm1} with $a(t) = \frac{d}{2t}$, $C=1$, $p=2$, and $r=1$. Since $r=p-1$, case (i) of the proof is relevant to this calculation. 

As before, let $\phi: \R^d_T\rightarrow \R$ and $W: \R^d \rightarrow \R^d$ be continuous functions. Then, we obtain from \eqref{eq:estimaterearranged} that
$$\frac{\partial u}{\partial t} + \frac{d}{2t}u - (\nabla u) \cdot W + \frac{1}{2} |\phi W|^2 \geq \frac{|\nabla u|^2}{u}- (\nabla u) \cdot W + \frac{1}{2} |\phi W|^2.$$
Choosing $\phi = (\frac{u}{2})^{1/2}$, the right-hand side of the last inequality can be expressed as 
\begin{equation}\label{eq:preyoung}
   \frac{2}{u} \left[ \frac{|\nabla u|^2}{2}-(\nabla u)\cdot\left(\frac{u}{2} W\right) + \frac{1}{2} \left(\frac{u}{2}\right)^2|W|^2 \right]. 
\end{equation}
Young's inequality \eqref{eq:young} with $p=q=2$, $Y=\nabla u$, and $Z = (\frac{u}{2})W$ yields that 
$$\frac{|\nabla u|^2}{2} + \frac{1}{2}\left(\frac{u}{2}\right)^2|W|^2 \geq (\nabla u) \cdot \left(\frac{u}{2} W\right).$$
Using this result and the nonnegativity of $u$, we conclude that the quantity \eqref{eq:preyoung} is nonnegative and hence
\begin{equation}\label{eq:nextinequality}
    \frac{\partial u}{\partial t} + \frac{d}{2t}u - (\nabla u) \cdot W + \frac{u}{4}|W|^2 \geq 0.
\end{equation}

Now, let $(x_1,t_1), (x_2,t_2) \in \R^d_T$ be arbitrary and let $x(t)$ be a $C^1$ path in $\R^d$ with $x(t_1) =x_1$ and $x(t_2) =x_2$. By the chain rule,
$$\frac{d}{dt}u(x(t),t) = (\nabla u) \cdot (\dot{x}(t)) + \frac{\partial u}{\partial t}.$$
Inserting this result into \eqref{eq:nextinequality} with $W=-\dot{x}$, we have that
\begin{align*}
    \frac{du}{dt} + \frac{d}{2t}u + \frac{u}{4}|\dot{x}|^2 &\geq 0\\
    \frac{1}{u}\frac{du}{dt} + \frac{d}{2t} &\geq -\frac{1}{4}|\dot{x}|^2.
\end{align*}
Integrating with respect to $t$ from $t_1$ to $t_2$ produces
$$\log(u(x_2,t_2)) + \frac{d}{2}\log\left(\frac{t_2}{t_1}\right) \geq -\frac{1}{4}\int_{t_1}^{t_2}|\dot{x}(t)|^2 \ \mathrm{d}t + \log(u(x_1,t_1)).$$
We optimise this inequality by minimising the integral $\frac{1}{2}\int_{t_1}^{t_2}|\dot{x}(t)|^2 \ \mathrm{d}t$. This minimisation problem is of the form described in Corollary \ref{cor:unweighted} with $q=2$. Therefore, the integral has minimum value 
$$\frac{|x_2-x_1|^2}{2(t_2-t_1
)}.$$
Hence, 
$$\log(u(x_2,t_2)) + \frac{d}{2}\log\left(\frac{t_2}{t_1}\right) \geq -\frac{|x_2 - x_1|^2}{4(t_2 - t_1)}+ \log(u(x_1,t_1)).$$
By exponentiating and rearranging, we obtain the inequality of Li and Yau,
$$u(x_2,t_2) \geq u(x_1,t_1) \left(\frac{t_1}{t_2} \right)^{d/2} e^{-\tfrac{|x_2-x_1|^2}{4(t_2-t_1)}},$$ which holds for all $x_1,x_2 \in \R^d$ and $0 < t_1 < t_2 < T$.

\section{The Porous Medium Equation}\label{PME}
Let $M > M_0(d):=\max\{ 0, 1-\frac{2}{d}\}$ and let $u$ be a positive solution of the porous medium equation,
\begin{equation*}
\frac{\partial u}{\partial t} = \Delta(u^M) \qquad \text{in } \R^d_T.
\end{equation*}
Considering first the case $M>1$, we make the change of dependent variable, $$
f := (\tfrac{M}{M-1})u^{M-1}.
$$ Then, $$\frac{\partial f}{\partial t} = \frac{df}{du} \frac{\partial u}{\partial t} = \frac{df}{du}\Delta(u^M).$$ By writing \begin{align*}
    \Delta(u^M) &= Mu^{M-1}\Delta u + M(M-1)u^{M-2}|\nabla u|^2\\ &= (M-1)f\Delta u + (M-1)\frac{df}{du}|\nabla u|^2
\end{align*}
we have that 
\begin{align*}
    \frac{\partial f}{\partial t} &= (M-1)f\frac{df}{du}\Delta u + (M-1)\left(\frac{df}{du}\right)^2|\nabla u|^2\\
    &= (M-1)f\frac{df}{du}\Delta u + (M-1)|\nabla f|^2.
\end{align*}
Then, using $\frac{df}{du}\Delta u= \Delta f - \frac{d^2f}{du^2}|\nabla u|^2$, it follows that
$$\frac{\partial f}{\partial t} = (M-1)f\Delta f + (M-1)|\nabla f|^2 - (M-1) f \frac{d^2 f}{du^2}|\nabla u|^2.$$
It is easily verified that $$ (M-1) f \frac{d^2 f}{du^2}|\nabla u|^2 =  (M-2)M^2u^{2M-4}|\nabla u|^2.$$ Recognising that $M^2u^{2M-4}|\nabla u|^2 =\left(\frac{df}{du}\right)^2|\nabla u|^2 =  |\nabla f|^2$ and simplifying, we arrive at the equation satisfied by $f$, which is
\begin{equation}
\frac{\partial f}{\partial t} = (M-1)f\Delta f + |\nabla f|^2. \label{eq:PMEf}
\end{equation}
These calculations are justified in a rigorous sense, since positive solutions of the porous medium equation with positive and continuous initial data are known to be smooth \cite{VazquezPME}.

To proceed further, we require the following significant and crucial inequality proven by Aronson and B\'enilan \cite{Aronson-Benilan-1979}. \begin{lemma}\label{lem:ab}
Let $M > M_0(d):=\max\{0,1-\frac{2}{d}\}$ and $u$ be a positive solution of the porous medium equation,
\begin{equation*}
\frac{\partial u}{\partial t} = \Delta(u^M) \qquad \text{in } \R^d_T.
\end{equation*} Define $$f=\begin{cases}\frac{M}{M-1} u^{M-1} & M \ne 1,\\ \log u & M=1. \end{cases}$$
Then \begin{equation}\label{eq:ab}
    \Delta f \geq \frac{-k}{t}, \qquad k:= \frac{1}{(M-1 + \frac{2}{d})}.
\end{equation}
\end{lemma}

By multiplying \eqref{eq:ab} by $(M-1)f$, we get that $$
(M-1)f\Delta f \geq - \frac{k(M-1)}{t}f.
$$ Then by \eqref{eq:PMEf}, \begin{equation}
\frac{\partial f}{\partial t}+\frac{(M-1)k}{t}f \geq |\nabla f|^2.\label{eq:inequ}
\end{equation}

Let $\mu:=(M-1)k$ and $\delta:=1-\mu$. Inequality \eqref{eq:inequ} is of the form \eqref{eq:thm1} with $a(t)=\frac{\mu}{t}$, $C=1$, $p=2$, $r=0$. With the goal of applying Theorem \ref{thm:generalharnackineq}, we calculate the remaining quantities appearing in this theorem as
\begin{equation}\label{eq:quantities}q=2, \quad m=-1, \quad \xi=\frac{1}{4}, \quad A(t)=\log(t^{\mu}), \quad I=\frac{t_2^{\delta}-t_1^{\delta}}{\delta}.\end{equation}
Inserting this information into inequality \eqref{eq:caseiiia}, we obtain that $f$ satisfies \begin{equation*}
f(x_2,t_2) \geq \left( \frac{t_1}{t_2} \right)^{\mu}\left[f(x_1,t_1) - \frac{\delta|x_2-x_1|^2}{4(t_2^{\delta} - t_1^{\delta})}\frac{1}{t_1^{\mu}}\right]
\end{equation*} for all $x_1,x_2 \in \R^d$ and $0<t_1<t_2<T$. Finally, writing this in terms of the original variable $u$, we arrive at the resulting Harnack inequality,
\begin{equation}\label{eq:harnackPME}
    [u(x_2,t_2)]^{M-1} \geq \left( \frac{t_1}{t_2} \right)^{\mu}\left[[u(x_1,t_1)]^{M-1} - \frac{M-1}{M}\frac{\delta|x_2-x_1|^2}{4(t_2^{\delta} - t_1^{\delta})}\frac{1}{t_1^{\mu}}\right].
\end{equation}

Next, we treat the case $M_0(d)<M<1$. Our previous definition of $f=\frac{M}{M-1}u^{M-1}$ is no longer appropriate, since the coefficient $\frac{M}{M-1}$ and hence also $f$ are now strictly negative. Thus, Theorem \ref{thm:generalharnackineq} cannot apply. Instead, define $$g:=\frac{M}{1-M}u^{M-1},$$ which will be nonnegative. By a similar process to that conducted above, $g$ satisfies the equation \begin{equation}\label{eq:g-equation}
    \frac{\partial g}{\partial t} = (1-M)g\Delta g - |\nabla g|^2 \qquad \text{in } \R^d_T.
\end{equation}
Observe that $g=-f$ and recall the inequality of Aronson and B\'enilan \cite{Aronson-Benilan-1979}, $\Delta f \geq -\frac{k}{t}$. Using the homogeneity of the Laplace operator and rearranging the inequality, we have that \begin{equation}\label{eq:g-ab}\Delta g \leq \frac{k}{t},\end{equation} where $k$ is defined as in Lemma \ref{lem:ab}. Multiplying \eqref{eq:g-ab} by $(1-M)g \geq 0$ and inserting the result into \eqref{eq:g-equation}, we have that $$\frac{\partial g}{\partial t} + \frac{(M-1)k}{t} g \leq -|\nabla g|^2.$$ This inequality is of the form \eqref{eq:thm4} with $a(t) = \frac{(M-1)k}{t}$, $C=1$, $p=2$, and $r=0$. We find that the remaining quantities required to apply Theorem \ref{thm:gen-backwards} coincide with those given in \eqref{eq:quantities} from the case $M>1$. Applying this theorem, we obtain that $g$ satisfies $$g(x_2,t_2) \leq \left( \frac{t_1}{t_2} \right)^{\mu}\left[g(x_1,t_1) + \frac{\delta|x_2-x_1|^2}{4(t_2^{\delta} - t_1^{\delta})}\frac{1}{t_1^{\mu}}\right]$$ for all $x_1,x_2 \in \R^d$ and $0<t_1<t_2<T$. Writing this in terms of $u$, we have $$[u(x_2,t_2)]^{M-1} \leq \left( \frac{t_1}{t_2} \right)^{\mu}\left[[u(x_1,t_1)]^{M-1} - \frac{M-1}{M}\frac{\delta|x_2-x_1|^2}{4(t_2^{\delta} - t_1^{\delta})}\frac{1}{t_1^{\mu}}\right].$$ Since both sides of the above inequality will be strictly positive, we may take their reciprocals in order to obtain a lower bound for $[u(x_2,t_2)]^{1-M}$. We prefer this form of the inequality to better show the analogy with \eqref{eq:harnackPME}. Doing this, we obtain
$$[u(x_2,t_2)]^{1-M} \geq \left( \frac{t_2}{t_1} \right)^{\mu}\left[[u(x_1,t_1)]^{M-1} - \frac{M-1}{M}\frac{\delta|x_2-x_1|^2}{4(t_2^{\delta} - t_1^{\delta})}\frac{1}{t_1^{\mu}}\right]^{-1},$$ for all $x_1,x_2 \in \R^d$ and $0<t_1<t_2<T$, which is the result claimed in Theorem \ref{thm:pme}.


\section{The p-diffusion Equation}\label{sec:pdiff}
Let $\frac{2d}{d+1}< p < \infty$ and suppose that $u$ is a positive weak solution of 
the $p$-diffusion equation,
\begin{equation*}
    \frac{\partial u}{\partial t} = \divergence(|\nabla u|^{p-2}\nabla u), \qquad \text {in } \R^d_T.
\end{equation*}
By this, we mean that $u \in W^{1,\infty}_{\loc}((0,T),L^2_{\loc}(\R^d)) \cap L^p_{\loc}((0,T), W^{1,p}_{\loc}(\R^d))$  satisfies 
\begin{equation}\label{eq:p-diff-weak}
    -\int_0^T \int_{\R^d} \Big( u \frac{\partial \varphi}{\partial t} - |\nabla u|^{p-2}\nabla u \nabla \varphi \Big) \ \mathrm{d}x \, \mathrm{d}t = 0
\end{equation}
for all test functions $\varphi \in C_c^{\infty}(\R^d_T)$. We aim to find a Harnack inequality satisfied by $u$.

For $p > 2$, let
\begin{equation*}
    f:=\frac{1}{\gamma}u^{\gamma}, \quad \gamma:= \frac{p-2}{p-1}.
\end{equation*}
Proceeding formally, we have that 
\begin{align*}
    \frac{\partial f}{\partial t} &= \frac{df}{du} \frac{\partial u}{\partial t}\\
    &=\frac{df}{du} \divergence(|\nabla u|^{p-2}\nabla u)\\
    &= u^{\gamma} (u^{-1}\divergence(|\nabla u|^{p-2}\nabla u)).
\end{align*}
Then, we see that
\begin{align*}
    & \phantom{= \ } \divergence(|\nabla f|^{p-2} \nabla f)\\
    &= \divergence\left(\left(\frac{df}{du}\right)^{p-1}|\nabla u|^{p-2} \nabla u\right) \\
    &= \sum_{i=1}^d \frac{\partial}{\partial x_i} \left[ \left( \frac{df}{du}\right)^{p-1} |\nabla u|^{p-2} \frac{\partial u}{\partial x_i}\right] \\
    &= \sum_{i=1}^d \left( \frac{df}{du}\right)^{p-1} \frac{\partial}{\partial x_i}\left(|\nabla u|^{p-2} \frac{\partial u}{\partial x_i}\right) + \sum_{i=1}^d \frac{\partial}{\partial x_i} \left[ \left( \frac{df}{du}\right)^{p-1}\right] |\nabla u|^{p-2} \frac{\partial u}{\partial x_i} \\
    &= \left( \frac{df}{du}\right)^{p-1} \divergence(|\nabla u|^{p-2} \nabla u) + |\nabla u|^{p-2}\sum_{i=1}^d \left( \frac{\partial u}{\partial x_i}\right)^2 (p-1) \left(\frac{df}{du}\right)^{p-2}\frac{d^2f}{du^2}\\
    &= \left( \frac{df}{du}\right)^{p-1} \divergence(|\nabla u|^{p-2} \nabla u) + |\nabla u|^p (p-1) \left(\frac{df}{du}\right)^{p-2}\frac{d^2f}{du^2} \\
    &= u^{-1} \divergence(|\nabla u|^{p-2} \nabla u) - u^{-2}|\nabla u|^p 
\end{align*}
Hence, we have
$$\frac{\partial f}{\partial t}= u^{\gamma}\divergence(|\nabla f|^{p-2}\nabla f)+u^{\gamma-2}|\nabla u|^p.$$
Noticing that $u^{\gamma} = \gamma f$ and $$|\nabla f|^p = u^{(\gamma -1)p}|\nabla u|^p=u^{\gamma -2}|\nabla u|^p,$$ we have that the equation satisfied by $f$ is
\begin{equation}
    \frac{\partial f}{\partial t} = \gamma f \divergence(|\nabla f|^{p-2}\nabla f) + |\nabla f|^p. \label{eq:newpdiff}
\end{equation}

Now, a weak solution $u$ of the $p$-diffusion equation  only belongs to the space $C^{1,\alpha}(\R^d_T)$ in general \cite{DiBenedetto-Friedman-1985}. A definition of this space may be found in Chapter \ref{sec:hoelder}. Thus, the above computations are merely formal. In order to study this problem rigorously, we must understand what is meant by a weak solution of \eqref{eq:newpdiff}. To do this, we propose the following definition.

\begin{definition}\label{def:our-def-f}
Let $1 < p < \infty$ and $\gamma \in \R$. Then, we call a function $f \in L^{\infty}_{\loc}(\R^d_T) \cap L^{p}_{\loc}((0,T), W^{1,p}_{\loc}(\R^d))$ a positive weak solution of  \begin{equation*}\tag{\ref{eq:newpdiff}}\frac{\partial f}{\partial f} = \gamma f \divergence(|\nabla f|^{p-2} \nabla f) + |\nabla f|^p \qquad \text{in } \R^d_T\end{equation*} if $f>0$ and $f$ satisfies 
\begin{align}\label{eq:fweakint}\begin{split}
&\int_0^T \int_{\R^d} f \frac{\partial \varphi}{\partial t} \ \mathrm{d}x \, \mathrm{d}t - \gamma \int_0^T \int_{\R^d} f|\nabla f|^{p-2}\nabla f \nabla \varphi \ \mathrm{d}x \, \mathrm{d}t \\ & \hspace{3.5cm}+  (1-\gamma)\int_0^T \int_{\R^d} |\nabla f|^p \varphi \ \mathrm{d}x \, \mathrm{d}t =0\end{split}
\end{align} for all $\varphi \in C_c^{\infty}(\R^d_T)$.
\end{definition}

To demonstrate that this notion of a solution is sensible, we show that for any positive weak solution $u$ of the $p$-diffusion equation, $f=\frac{1}{\gamma} u^{\gamma}$ with $\gamma=\frac{p-2}{p-1}$ satisfies our definition. To achieve this, we require some properties of weak solutions to the $p$-diffusion equation, which we summarise without proof in the proposition below.

\begin{proposition}\label{prop:pdiffprops} Let $u$ be a positive weak solution of the $p$-diffusion equation. Suppose that $u_0:=u(\cdot,0) \in L^1(\R^d)$. Then, the following properties hold.

\begin{enumerate}
    \item (Existence of weak time derivative, \cite{Barbu}): $u$ is weakly differentiable with respect to time and $$\frac{\partial u}{\partial t} \in L^2_{\loc}(\R^d_T);$$
    \item (Gradient regularity), \cite{Alikakos}): One has that $$\nabla u(t) \in (L^{\infty}(\R^d))^d$$ for all $t > 0$;
    \item ($L^1$-$L^{\infty}$ estimate, \cite{Coulhon-Hauer}): There exists a constant $C>0$ such that $$\|u(t)\|_{\infty} \leq Ct^{-\alpha_1}\|u_0\|_1^{\gamma_1}$$ for all $t >0$, where $\alpha_1$ and $\gamma_1$ are constants depending on $d$ and $p$. Thus, $u(t) \in L^{\infty}(\R^d)$ for all $t>0$.
\end{enumerate}
\end{proposition}

First, we note that since $u>0$, for every $K \subset \subset \R^d$, there exists a constant $c_{K} >0$ such that $u \geq c_{K}$ a.e. on $K$. Then, since the $L^1$-$L^{\infty}$ estimate (iii) from Proposition \ref{prop:pdiffprops} implies that $u(t) \in L^{\infty}(\R^d)$ for all $t>0$, it follows that $f=\frac{1}{\gamma} u^{\gamma} \in L^{\infty}_{\loc}(\R^d_T)$. Similarly, we have that \mbox{$u(t)^{\gamma-1} \in L^{\infty}(\R^d)$} for all $t>0$. Combining this with $\nabla u (t) \in (L^{\infty}(\R^d))^d$ from (ii) in Proposition \ref{prop:pdiffprops}, we see that $$\nabla f(t) = u(t)^{\gamma-1}\nabla u(t) \in (L^{\infty}_{\loc}(\R^d))^d \hookrightarrow (L^{p}_{\loc}(\R^d))^d$$ for all $t>0$. Thus, $f \in L^{\infty}_{\loc}(\R^d_T) \cap L^{p}_{\loc}((0,T), W^{1,p}_{\loc}(\R^d))$.

Now, in order to write $\nabla f =\nabla(\frac{1}{\gamma}u^{\gamma}) = u^{\gamma-1}\nabla u$ in the step above, we required the chain rule. To justify this, we first observe that for any $K \subset \subset \R^d_T$, there exist constants $0< c_K \leq c'_{K} < \infty$ such that $c_K \leq u \leq c'_K$ a.e. on $K$. Then, the function $G(s):=\frac{1}{\gamma}s^{\gamma}$ is continuously differentiable with bounded derivative $G'$ on the interval $[c_K,c'_{K}]$ for all $\gamma \in \R$. Extending $G$ to all $\R$ in such a way that $G(0)=0$, the continuous differentiability of $G$ is preserved, and the boundedness of $G'$ is preserved, we may conclude by Proposition 9.5 in \cite{Brezis} that $$\frac{\partial}{\partial t}(\tfrac{1}{\gamma}u^{\gamma}) = u^{\gamma-1}\frac{\partial u}{\partial t}, \qquad \nabla (\tfrac{1}{\gamma}u^{\gamma}) = u^{\gamma-1}\nabla u.$$

Next, we fix $\varphi \in C_c^{\infty}(\R^d_T)$ and insert $f=\frac{1}{\gamma}u^{\gamma}$ into the right-hand side of \eqref{eq:fweakint} to obtain
\begin{align*}
    &\int_0^T \int_{\R^d} f \frac{\partial \varphi}{\partial t} \ \mathrm{d}x \, \mathrm{d}t - \gamma \int_0^T \int_{\R^d} f|\nabla f|^{p-2}\nabla f \nabla \varphi \ \mathrm{d}x \, \mathrm{d}t  \\ & \hspace{1cm}+ (1-\gamma)\int_0^T \int_{\R^d} |\nabla f|^p \varphi \ \mathrm{d}x \, \mathrm{d}t \\
&= \frac{1}{\gamma}\int_0^T \int_{\R^d} u^{\gamma} \frac{\partial \varphi}{\partial t} \ \mathrm{d}x \, \mathrm{d}t - \gamma \int_0^T \int_{\R^d} u^{\gamma-1}|\nabla u|^{p-2}\nabla u \nabla \varphi \ \mathrm{d}x \, \mathrm{d}t  \\ & \hspace{1cm}- (\gamma-1)\int_0^T \int_{\R^d} u^{\gamma-2}|\nabla u|^p \varphi \ \mathrm{d}x \, \mathrm{d}t \\
&=\frac{1}{\gamma}\int_0^T \int_{\R^d} u^{\gamma}\frac{\partial \varphi}{\partial t} \ \mathrm{d}x \, \mathrm{d}t - \int_0^T \int_{\R^d} |\nabla u|^{p-2}\nabla u \nabla(u^{\gamma-1}\varphi) \ \mathrm{d}x \, \mathrm{d}t,
\end{align*}
where here we have used the chain rule and the product rule to obtain the last equality. The chain rule is justified as before. To justify the product rule, we let $\omega:=\supp \varphi \subset \subset \R^d_T$ and observe that $u^{\gamma-1}, \varphi \in L^{\infty}(\omega) \cap W^{1,p}(\omega)$. Then, by Proposition 9.4 in \cite{Brezis}, $$\nabla (u^{\gamma-1}\varphi) = \nabla(u^{\gamma-1})\varphi + u^{\gamma-1}\nabla \varphi.$$ Since all the terms in the above equation are identically zero on $\R^d_T \setminus \omega$, the product rule is still valid when we extend the domain of these functions from $\omega$ to $\R^d_T$.

Using the chain rule and integrating by parts twice, we rewrite the following integral.
\begin{align*}\frac{1}{\gamma}\int_0^T \int_{\R^d} u^{\gamma}\frac{\partial \varphi}{\partial t} \ \mathrm{d}x \, \mathrm{d}t &= -\int_0^T \int_{\R^d} \frac{\partial u}{\partial t} u^{\gamma-1} \varphi \ \mathrm{d}x \, \mathrm{d}t \\&= \int_0^T \int_{\R^d} u \frac{\partial}{\partial t} (u^{\gamma-1} \varphi)\ \mathrm{d}x \, \mathrm{d}t.\end{align*} Here, we have also used the existence of $\frac{\partial u}{\partial t}$ in the weak sense (Proposition \ref{prop:pdiffprops}, (i)) to write the integral in the intermediate step. Thus, the expression obtained from the right-hand side of \eqref{eq:fweakint} is equal to
\begin{equation}\label{eq:the result}
    \int_0^T \int_{\R^d} u \frac{\partial}{\partial t} (u^{\gamma-1} \varphi) \ \mathrm{d}x \, \mathrm{d}t - \int_0^T \int_{\R^d} |\nabla u|^{p-2}\nabla u \nabla(u^{\gamma-1}\varphi) \ \mathrm{d}x \, \mathrm{d}t.
\end{equation}
It now remains to show that the expression \eqref{eq:the result} is equal to zero. Set $\psi:=u^{\gamma-1}\varphi$ and write \eqref{eq:the result} as 
\begin{equation}\label{eq:the psi one}
    \int_{\omega} u \frac{\partial \psi}{\partial t}  \ \mathrm{d}(x,t) -  \int_{\omega} |\nabla u|^{p-2}\nabla u \nabla \psi \ \mathrm{d}(x,t).
\end{equation}
Since we know from earlier that $u^{\gamma-1} \in L^{\infty}_{\loc}(\R^d_T)$ and $\varphi \in L^{\infty}(\R^d_T)$, we have that $\psi = u^{\gamma-1}\varphi \in L^{\infty}(\omega) \subseteq L^1(\omega)$. As well, using the product rule and chain rule, $$\frac{\partial \psi}{\partial t} = (\gamma-1)u^{\gamma-2}\frac{\partial u}{\partial t}\varphi + u^{\gamma-1}\frac{\partial \varphi}{\partial t}.$$ The functions $u^{\gamma-2}$ and $u^{\gamma-1}$ both belong to $L^{\infty}_{\loc}(\R^d_T)$, $\frac{\partial u}{\partial t} \in L^2_{\loc}(\R^d_T)$ by (i) in Proposition \ref{prop:pdiffprops}, and $\varphi, \frac{\partial \varphi}{\partial t} \in L^{\infty}(\R^d_T)$. Thus, $\frac{\partial \psi}{\partial t} \in L^2(\omega) \subseteq L^1(\omega)$. Similarly, $$\nabla \psi = (\gamma-1)u^{\gamma-2}(\nabla u)\varphi + u^{\gamma-1}\nabla \varphi.$$ Repeating the previous argument, instead using that $\nabla u \in L^{\infty}_{\loc}(\R^d_T)$ by (ii) in Proposition \ref{prop:pdiffprops}, we have that $\nabla \psi \in (L^{\infty}(\omega))^d \subseteq (L^1(\omega))^d$. Thus, $\psi \in W^{1,1}(\omega)$. However, since $\psi\equiv 0$ on $\R^d_T \setminus \omega$, we may conclude $$\psi \in W^{1,1}_0(\R^d_T):=\overline{C_c^{\infty}(\R^d_T)}^{\| \cdot \|_{W^{1,1}}}.$$
Hence, there exists a sequence $(\psi_n)_{n \geq 1}$ in $C_c^{\infty}(\R^d_T)$ such that $\psi_n \rightarrow \psi$ in $W^{1,1}(\R^d_T)$. This implies that $$\frac{\partial \psi_n}{\partial t} \rightarrow \frac{\partial \psi}{\partial t} \text{ in $L^1(\R^d_T),$}\qquad \nabla \psi_n \rightarrow \nabla \psi \text{ in $(L^1(\R^d_T))^d$.}$$ Additionally, since $\omega \subseteq \R^d_T$, these convergences occur in $L^1(\omega)$ and $(L^1(\omega))^d$ respectively as well.

Now, since $\psi_n \in C_c^{\infty}(\R^d_T)$, the definition \eqref{eq:p-diff-weak} of $u$ as a weak solution of the $p$-diffusion equation implies
$$\int_0^T \int_{\R^d}  u \frac{\partial \psi_n}{\partial t} \ \mathrm{d}x \, \mathrm{d}t- \int_0^T \int_{\R^d} |\nabla u|^{p-2}\nabla u \nabla \psi_n \ \mathrm{d}x \, \mathrm{d}t = 0$$
for all $n\geq 1$. Since $\frac{\partial \psi}{\partial t}, \nabla \psi_n$ are both identically zero on $\R^d_T \setminus \omega$, we may rewrite this as
$$\int_{\omega}  u \frac{\partial \psi_n}{\partial t} \ \mathrm{d}(x,t)- \int_{\omega} |\nabla u|^{p-2}\nabla u \nabla \psi_n \ \mathrm{d}(x,t) = 0$$ for all $n \geq 1$. Since $u \in L^{\infty}(\omega)$ and $|\nabla u|^{p-2}\nabla u \in (L^{\infty}(\omega))^d$, we may use the duality $L^{\infty}(\omega) \cong (L^1(\omega))'$ to write $$\langle u, \frac{\partial \psi_n}{\partial t}\rangle_{L^{\infty},L^1} - \langle |\nabla u|^{p-2}\nabla u, \nabla \psi_n \rangle_{L^{\infty},L^1}=0.$$ Taking a limit as $n \rightarrow \infty$ in both sides, we obtain $$\langle u, \frac{\partial \psi}{\partial t}\rangle_{L^{\infty},L^1} - \langle |\nabla u|^{p-2}\nabla u, \nabla \psi \rangle_{L^{\infty},L^1}=0,$$ which is equivalent to \eqref{eq:the psi one} being equal to zero. This concludes the proof of our claim that $f=\frac{1}{\gamma}u^{\gamma}$ satisfies Definition \ref{def:our-def-f}.

\begin{remark}
In existing literature, for example, the work of Esteban and V\'azquez \cite{Esteban-Vazquez-1988, Esteban-Vasquez-1990}, a weak solution of \eqref{eq:newpdiff} has been understood as the limit as $\varepsilon \rightarrow 0$  of a sequence of solutions to the regularised equation \begin{equation}\label{eq:regular}
    \frac{\partial f}{\partial t}  = \gamma f \divergence(\varphi_{\varepsilon} (\nabla f) \nabla f) + \psi_{\varepsilon}(\nabla f)
\end{equation}
where $\varphi_{\varepsilon}$ and $\psi_{\varepsilon}$ are smooth approximations of the nonlinearities in the original equation. The solutions of \eqref{eq:regular} enjoy the property of being smooth, however this property is not in general transferred to the solutions of \eqref{eq:newpdiff} obtained in the limit. For our purposes, we require a more explicit understanding of the regularity properties of the solutions, hence, we introduce Definition \ref{def:our-def-f}.
\end{remark}

We now return our attention to deriving a Harnack inequality for $u$. Following Aronson and B\'enilan's discovery of inequality \eqref{eq:ab} involving solutions of the porous medium equation, Esteban and V\'azquez \cite{Esteban-Vasquez-1990} found an analogous inequality in the case of the $p$-diffusion equation. This result is essential to our proof and we state it in the lemma below.

\begin{lemma}\label{lem:estvaz}
Let $\frac{2d}{d+1}< p < \infty$ and let $u$ be a positive solution of the p-diffusion equation,
\begin{equation*}\frac{\partial u}{\partial t} = \divergence(|\nabla u|^{p-2}\nabla u) \qquad \text{in }\R^d_T.\end{equation*} Define $$f=\begin{cases}\frac{1}{\gamma}u^{\gamma} & p \ne 2, \\
\log u & p=2\end{cases}$$ where $\gamma:=\frac{p-2}{p-1}$. Then, the inequality
\begin{equation}\label{eq:estvaz}
    \divergence(|\nabla f|^{p-2}\nabla f) \geq -\frac{K}{t} \qquad \text{in }\mathscr{D}'(\R^d_T)
\end{equation} holds. 
\end{lemma}
In \eqref{eq:estvaz}, we write $\mathscr{D}'(\R^d_T)$ in order to emphasise that this inequality holds in the sense of distributions. By this, we mean that we interpret $T=\divergence(|\nabla f|^{p-2}\nabla f)$ to be a distribution given by $$\langle T,\varphi \rangle_{\mathscr{D}',\mathscr{D}} = -\int_0^T \int_{\R^d} |\nabla f|^{p-2}\nabla f \nabla \varphi \ \mathrm{d}x \, \mathrm{d}t$$ for all test functions $\varphi \in \mathscr{D}(\R^d_T)$. For a brief explanation of the fundamental ideas in the theory of distributions, we refer the reader to Appendix \ref{sec:distributions}.

Multiplying \eqref{eq:estvaz} by $\gamma f$, gives $$\gamma f\divergence(|\nabla f|^{p-2}\nabla f) \geq\left(\frac{\gamma  K}{t}\right)f \qquad \text{ in }\mathscr{D}'(\R^d_T)$$
Using this result with the equation \eqref{eq:newpdiff} yields that $$\frac{\partial f}{\partial t} - |\nabla f|^p \geq -\left(\frac{\gamma  K}{t}\right) f \qquad \text{in }\mathscr{D}'(\R^d_T)$$ or equivalently, \begin{equation}\frac{\partial f}{\partial t} + \left(\frac{\gamma  K}{t}\right) f \geq |\nabla f|^p \qquad \text{in }\mathscr{D}'(\R^d_T).\label{eq:ineq6}\end{equation}
Since $f, \frac{\partial f}{\partial t}, |\nabla f|^p \in L^1_{\loc}(\R^d_T)$, the distributions appearing in \eqref{eq:ineq6} are all regular distributions. Thus, we may interpret this inequality as holding pointwise a.e. in $\R^d_T$.

We observe that \eqref{eq:ineq6} is of the form \eqref{eq:thm1} with $C=1$, $r=0$, and $a(t) = \frac{\gamma K}{t}$. Since $u>0$ is Lipschitz continuous \cite{Alikakos}, $f=\frac{1}{\gamma}u^{\gamma}$ is also Lipschitz continuous and thus satisfies the assumptions (F1)--(F3) stated in Chapter \ref{sec:genharnack}. Since $m = \frac{r}{p-1}=-1$, we may apply case (iii) of Theorem \ref{thm:gen-dist}. We calculate that $\xi = \frac{1}{q}\left(\frac{1}{p}\right)^{q-1}$ and $e^{A(t)} = e^{\int^t \frac{\gamma k}{s} }\ \mathrm{d}s = t^{\gamma K}$. We must also calculate the value of $$I = \int_{t_1}^{t_2} e^{m(p-1)A(t)} \ \mathrm{d}t = \int_{t_1}^{t_2} t^{-\gamma K(p-1)} \ \mathrm{d}t,$$ which depends on the value of $1-\gamma K(p-1) = (2-p)K+1 =: \delta$. In particular, one has 
\begin{equation}\label{eq:integral1}
    I=\begin{cases}\dfrac{t_2^{\delta}-t_1^{\delta}}{\delta} & \delta \ne 0, \\ \log t_2 - \log t_1 & \delta =0.\end{cases}
\end{equation}
Inserting this information into \eqref{eq:caseiiia} gives 
\begin{equation*}
    f(x_2,t_2) \geq \left( \frac{t_1}{t_2} \right)^{\gamma K} \left( f(x_1,t_1) - \xi|x_2-x_1|^qI^{1-q}t_1^{-\gamma K} \right).
\end{equation*}
Finally, using that $f=\frac{1}{\gamma}u^{\gamma}$, we obtain the following inequality for $u$, 
\begin{equation*}
    [u(x_2,t_2)]^{\gamma} \geq \left( \frac{t_1}{t_2} \right)^{\gamma K} \left( [u(x_1,t_1)]^{\gamma} -\gamma \xi|x_2-x_1|^qI^{1-q}t_1^{-\gamma K} \right),
\end{equation*}
for all $x_1,x_2 \in \R^d$ and $0 < t_1 < t_2 < T$.

We now consider the case $\frac{2d}{d+1} < p < 2$. For $p$ in this range, $\gamma = \frac{p-2}{p-1} < 0$ and therefore, $f=\frac{1}{\gamma}u^{\gamma}$ is nonpositive. Thus, we redefine our transformation by introducing $$g= - \frac{1}{\gamma}u^{\gamma}$$ with $\gamma$ defined as before. Importantly we have that $g \geq 0$ a.e. in $\R^d_T$. By a similar derivation used to obtain \eqref{eq:newpdiff}, $g$ formally satisfies the equation 
\begin{equation}\label{eq:newpdiffg}
    \frac{\partial g}{\partial t} + \gamma g \divergence(|\nabla g|^{p-2} \nabla g) = -|\nabla g|^p.
\end{equation}

Similar to our Definition \ref{def:our-def-f} of a weak solution of the equation \eqref{eq:newpdiff} satisfied by $f$, we understand a weak solution of \eqref{eq:newpdiffg} by the following.

\begin{definition}\label{def:our-def-g}
Let $1 < p < \infty$ and $\gamma \in \R$. Then, we call a function $g \in L^{\infty}_{\loc}(\R^d_T) \cap L^{p}_{\loc}((0,T), W^{1,p}_{\loc}(\R^d))$ a positive weak solution of  \begin{equation*}\tag{\ref{eq:newpdiffg}}\frac{\partial g}{\partial t} + \gamma g \divergence(|\nabla g|^{p-2} \nabla g) = - |\nabla g|^p \qquad \text{in } \R^d_T\end{equation*} if $g >0$ and $g$ satisfies
\begin{align*}\label{eq:gweakint}\begin{split}
&-\int_0^T \int_{\R^d} g \frac{\partial \varphi}{\partial t} \ \mathrm{d}x \, \mathrm{d}t - \gamma \int_0^T \int_{\R^d} g|\nabla g|^{p-2}\nabla g \nabla \varphi \ \mathrm{d}x \, \mathrm{d}t \\& \hspace{3.8cm} + (1-\gamma)\int_0^T \int_{\R^d} |\nabla g|^p \varphi \ \mathrm{d}x \, \mathrm{d}t =0\end{split}
\end{align*} for all $\varphi \in C_c^{\infty}(\R^d_T)$.
\end{definition}

In particular, we note that by a similar argument used to justify Definition \ref{def:our-def-f}, the function $g=-\frac{1}{\gamma}u^{\gamma}$ will satisfy Definition \ref{def:our-def-g}.

Our new change of variables was set up so that $g=-f$. Applying the inequality \eqref{eq:estvaz} of Esteban and V\'azquez to $f$ and rearranging, we find that $g$ satisfies $$\divergence (|\nabla g|^{p-2}\nabla g) \leq \frac{K}{t} \qquad \text{in }\mathscr{D}'(\R^d_T)$$ where $K:=K(p,d)$ is defined as in Lemma \ref{lem:estvaz}. Multiplying this result by $\gamma g \leq 0$, we have $$\gamma g \divergence (|\nabla g|^{p-2}\nabla g) \geq \frac{\gamma K}{t} \qquad \text{in }\mathscr{D}'(\R^d_T)$$ and by using \eqref{eq:newpdiffg}, we see that $g$ satisfies $$\frac{\partial g}{\partial t} + \frac{\gamma K}{t}g \leq  -|\nabla g|^p \qquad \text{in }\mathscr{D}'(\R^d_T).$$ This inequality is of the form \eqref{eq:thm4} with $C=1$, $r=0$, and $a(t)=\frac{\gamma K}{t}$. As was calculated earlier in the case $p>2$, this leads to the quantities $q=\frac{p}{p-1}$, $m=-1$, $\xi = \frac{1}{q}(\frac{1}{p})^{q-1}$, $e^{A(t)}=t^{\gamma K}$, and $I$ given by \eqref{eq:integral1}. Applying case (iii) of Theorem \ref{thm:gen-backwards}, we have that $g$ satisfies
$$g(x_2,t_2) \leq \left( \frac{t_1}{t_2} \right)^{\gamma K} \left( g(x_1,t_1) + \xi|x_2-x_1|^qI^{1-q}t_1^{-\gamma K} \right)$$ for all $x_1,x_2 \in \R^d$ and $0 < t_1 < t_2 < T$. Returning to the original variables using $g=-\frac{1}{\gamma}u^{\gamma}$, we have that $$[u(x_2,t_2)]^{\gamma} \leq \left( \frac{t_1}{t_2} \right)^{\gamma K} \left( [u(x_1,t_1)]^{\gamma} -\gamma \xi|x_2-x_1|^qI^{1-q}t_1^{-\gamma K} \right).$$ Finally, we may take the reciprocal of both sides of  this inequality since both sides are strictly positive. Hence, we have that $g$ satisfies
\begin{equation*}
        [u(x_2,t_2)]^{-\gamma} \geq \left( \frac{t_2}{t_1} \right)^{\gamma K} \left( [u(x_1,t_1)]^{\gamma} -\gamma \xi|x_2-x_1|^qI^{1-q}t_1^{-\gamma K} \right)^{-1}
    \end{equation*}
for all $x_1,x_2 \in \R^d$ and $0 < t_1 < t_2 < T$, which is the result claimed in Theorem \ref{thm:pdiff}. 

\begin{remark}\label{rem:pdiffsubcritical}
The techniques used to derive Harnack inequalities for solutions of the $p$-diffusion equation were limited to the case $ \frac{2d}{d+1}< p < \infty$. In our proof, this restriction arose from the fact that the result of Esteban and V\'azquez (Lemma \ref{lem:estvaz}) does not hold for $p$ in the range, $1 < p \leq \frac{2d}{d+1}$, which is sometimes referred to as the \emph{subcritical range} in literature. However, it is known from the work of DiBenedetto et al. \cite{DiBenedettoGianazzaVespri} that in general, a Harnack inequality of the form discussed in this monograph does not hold for the solutions corresponding to $p$ in this subcritical range. 
\end{remark}

\chapter{Harnack Inequalities and H\"older Continuity}\label{sec:hoelder}

A standard application of Harnack inequalities is proving that the solutions of a parabolic equation are H\"older continuous jointly in space and time. As a demonstration, we will explore the proof introduced by Moser \cite{Moser1964} showing that nonnegative weak solutions of the linear parabolic equation  \begin{equation*}\frac{\partial u}{\partial t} = \sum_{k,l = 1}^{n} \frac{\partial}{\partial x_k}\left( a_{kl}(x,t)\frac{\partial u}{\partial x_l}\right) \qquad \text{in } \Omega_T \tag{\ref{eq:para}}\end{equation*} are locally H\"older continuous in $\Omega_T$. As in Chapter \ref{sec:intro}, the coefficients $a_{kl}(x,t)$ are assumed to be measurable functions with $a_{kl} = a_{lk}$ that satisfy the uniform ellipticity condition \eqref{eq:ellipticity}. In addition, we will make use of information about the explicit values of the constants appearing in the proof, which has been contributed recently by Bonforte et al. \cite{Bonforte}. However, we first introduce some key definitions for this chapter.

\begin{definition}[Uniform and local H\"older continuity, \cite{GilbargTrudinger}] Let $u: \Omega \rightarrow \R$ be a function on a bounded set $\Omega \subseteq \R^d$. The function $u$ is called \emph{uniformly H\"older continuous} in $\Omega$ with exponent $\alpha \in (0,1)$ if 
\begin{equation*}\label{eq:hoelder definition}
    \sup_{\substack{x,y \in \Omega \\ x\ne y}} \frac{|u(x)-u(y)|}{|x-y|^{\alpha}} < \infty.
\end{equation*}
As well, we call $u$ \emph{locally H\"older continuous} in $\Omega$ if $u$ is uniformly H\"older continuous on all compact subsets of $\Omega$.\end{definition}

Moreover, one may define the \emph{H\"older space} $C^{k,\alpha}(\Omega)$ for all nonnegative integers $k$ and $\alpha \in (0,1)$ as the space of functions $u \in C^k(\Omega)$, such that all $k^{\text{th}}$-order partial derivatives of $u$ are uniformly H\"older continuous on $\Omega$. Finally, we note that the property of being H\"older continuous implies continuity in the usual sense.

We require the following domains, which we define for all $x_0 \in \R^d$, $t_0 > 0$ and $R>0$:
\begin{align*}
    D_R(x_0,t_0)&:= B_{2R}(x_0) \times (t_0 - R^2, t_0 + R^2)\\
    D_R^+(x_0,t_0)&:= B_{R/2}(x_0) \times (t_0 +\tfrac{3}{4} R^2, t_0 + R^2)\\
    D_R^-(x_0,t_0)&:= B_{R/2}(x_0) \times (t_0 - \tfrac{3}{4}R^2, t_0 - \tfrac{1}{4}R^2)
\end{align*}
With these domains, we may state the Harnack inequality of Moser \cite{Moser1964} by the following.

\begin{theorem}\label{thm:moserharnack}
Let $T>0$, $R \in (0,\sqrt{T})$ and let $(x_0,t_0) \in \Omega_T$ be such that $D_R(x_0,t_0) \subset \Omega_T$. Then, there exists a constant $C>1$ such that 
\begin{equation}\label{eq:moserharnack}
    \sup_{D_R^-(x_0,t_0)} u \leq C\inf_{D_R^+(x_0,t_0)} u
\end{equation} for all nonnegative weak solutions $u$ of \eqref{eq:para}.
\end{theorem}
We note that the constant $C$ in this theorem was shown in \cite{Bonforte} to be at least $\tfrac{4}{3}$. For more details concerning the explicit value of this constant, we refer the reader to \cite{Bonforte}.

Next, we suppose that $\Omega' \subset \Omega \subseteq \R^d$ are bounded domains. In what will follow, we will also assume for simplicity that $\Omega,\Omega'$ are convex domains, although this assumption is not necessary. Then, define $Q:=\Omega \times (T_1,T_4)$ and $Q':=\Omega' \times (T_2,T_3)$, where $0 < T_1 < T_2 < T_3 < T_4 < T$. Finally, we define the parabolic distance between $Q$ and $Q'$ by $$d(Q,Q'):=\inf_{\substack{(x,t) \in Q' \\ (y,s) \in \partial \Omega \times [T_1,T_4] \cup \{T_1,T_4\} \times \Omega}} |x-y| + |t-s|^{1/2}.$$ Now, we are ready to state the main result of this section, which we formulate as in \cite{Bonforte}.

\begin{theorem}
Suppose $u$ is a bounded solution of \eqref{eq:para} on $Q$. Then \begin{equation}\label{eq:hoelder}
    \sup_{(x,t),(y,s) \in Q'} \frac{|u(x,t)-u(y,s)|}{(|x-y| + |t-s|^{1/2})^{\nu}} \leq 2 \left(\frac{256}{d(Q,Q')} \right)^{\nu} \|u\|_{L^{\infty}(Q)}
\end{equation} where $\nu :=\log_4 (\frac{C}{C-1})$ and $C$ is the constant appearing in Theorem \ref{thm:moserharnack}. In particular, this implies that $u$ is locally H\"older continuous on $Q$.
\end{theorem}

\begin{proof}

The proof will occur in two main steps. First, we will compare the oscillation of $u$ on $D_R(x_0,t_0)$ and $D_R^{+}(x_0,t_0)$, resulting in an inequality with a quantitative bound. By the oscillation of $u$ on a domain, we mean the difference of the supremum and infimum of $u$ on this domain. Then, we will construct a finite sequence of nested sets, on which we will iterate the inequality found in the first step.

First, define the quantities \begin{align*}
M&:= \sup_{D_R(x_0,t_0)} u, & m&:= \inf_{D_R(x_0,t_0)} u, & \omega&:= M-m, \\ M^{\pm}&:= \sup_{D_R^{\pm}(x_0,t_0)} u,    & m^{\pm}&:= \inf_{D_R^{\pm}(x_0,t_0)} u, & \omega^{+}&:= M^{+} - m^{+}.
\end{align*}

By construction, the functions $M-u$ and $u-m$ will be nonnegative solutions of \eqref{eq:para} and will hence satisfy the Harnack inequality \eqref{eq:moserharnack}. Thus, one has that
\begin{align*}
    M-m^- = \sup_{D_R^-(x_0,t_0)}(M-u) &\leq C \inf_{D_R^+(x_0,t_0)} (M-u) = C(M-M^+)\\
    M^- -m = \sup_{D_R^-(x_0,t_0)}(u-m) &\leq C \inf_{D_R^+(x_0,t_0)} (u-m) = C(m^+ -m)
\end{align*}
Adding these two inequalities, we obtain
$$M-m^- + M^- - m \leq C(M-M^+ + m^+ - m).$$ Using that $M^- - m^- \geq 0$, we may write the following inequality in terms of the oscillations $\omega, \omega^+$, 
$$\omega \leq C(\omega - \omega^+),$$ which rearranges to give \begin{equation}\label{eq:stageonehoelder}
    \omega^+ \leq \frac{C-1}{C} \omega,
\end{equation}
concluding the first stage of the proof. For later use, we let $\zeta:=\frac{C-1}{C}$.

Next, we let $\delta:=\frac{d(Q,Q')}{64}$, so that if $(x,t) \in Q'$ and $(y,s)$ is any point in $\R^d_T$, then $$|x-y| + |t-s|^{1/2} \leq \delta $$ implies $(y,s) \in Q$. In particular, this inequality guarantees that $(y,s)$ cannot be far away from the set $Q'$. Now, for any $(x,t), (y,s) \in Q'$, one of the two following inequalities must hold:
\begin{align}\label{eq:caseihoelder} |x-y| + |t-s|^{1/2} &< \delta\\
\label{eq:caseiihoelder} |x-y| + |t-s|^{1/2} &\geq \delta
\end{align}
Suppose first that $\eqref{eq:caseihoelder}$ holds. Then there exists a nonnegative integer $k$ such that 
\begin{equation}\label{eq:kinteger}
    \frac{\delta}{4^{k+1}} \leq |x-y| + |t-s|^{1/2} \leq \frac{\delta}{4^k}.
\end{equation}

Next, let $z:=\frac{x+y}{2}$ and $\tau_0:=\frac{t+s}{2}$. We note that since we assumed that $Q'$ is convex, we have $(z,\tau_0) \in Q'$. Then, we construct a finite sequence of nested domains determined by the following values for all $0 \leq j\leq k-1$.
\begin{align*}
    R_{j+1}&:=4R_j, \quad R_0 = \frac{\delta}{4^{k-1}}\\
    \tau_{j+1}&:=\tau_j - 14 R_j^2, \quad \tau_0:=\frac{t+s}{2}
\end{align*}
In particular, $D_{R_j}(z,\tau_j) \subset D_{R_{j+1}}^+(z,\tau_{j+1})$ and $D_{R_0}(z,\tau_0) \subset D_{R_1}^+ (z,\tau_1)$. Indeed, we observe that the definition of $R_j$ implies that $B_{2R_j}(z) = B_{(R_{j+1})/2}$. This, together with $$\tau_{j+1} + \tfrac{3}{4} R_{j+1}^2 = \tau_j - 14R_j^2 + \tfrac{3}{4}\cdot 16 R_j^2 = \tau_j - 2R_j^2 < \tau_j - R_j^2$$ and $$\tau_{j+1} + R_{j+1}^2 = \tau_j - 14R_j^2 + 16R_j^2 = \tau_j + 2R_j^2 > \tau_j + R_j^2$$ implies that $D_{R_j}(z,\tau_j) \subset D_{R_{j+1}}^+(z,\tau_{j+1})$ for all $0 \leq j\leq k-1$. As well, using the right inequality in \eqref{eq:kinteger},
$$|x-z| = \frac{1}{2}|x-y| \leq \frac{1}{2}\left(\frac{\delta}{4^k} - |t-s|^{1/2}\right)  \leq \frac{\delta}{4^{k-1}}=R_0$$ and $$|t-\tau_0| =\frac{1}{2}\left( \frac{\delta}{4^k} - |x-y|\right)^2 \leq \left( \frac{\delta}{4^{k-1}}\right)^2$$ and so $$\tau_0 - R_0^2 \leq t \leq \tau_0 + R_0^2.$$ By repeating this process with $(x,t)$ replaced by $(y,s)$, we are able to conclude $(x,t),(y,s) \in D_{R_0}(z,\tau_0)$. Finally, we observe that $D_{R_k}(z,\tau_k) \subset Q$ due to the property obtained from the definition of $\delta$. 

In preparation for iterating the inequality obtained in the first stage of the proof we define the oscillations
\begin{align*}
    \omega_j&:=\sup_{D_{R_j}(z,\tau_j)} u - \inf_{D_{R_j}(z,\tau_j)}u,\\\omega_j^+&:=\sup_{D_{R_j}^+(z,\tau_j)} u - \inf_{D_{R_j}^+(z,\tau_j)}u.
\end{align*}
By the properties of the oscillation of a function, $D_{R_j}(z,\tau_j) \subset D_{R_{j+1}}^+(z,\tau_{j+1})$ implies that 
\begin{equation}\label{eq:simpleoscillation}
    \omega_j \leq \omega_{j+1}^+
\end{equation} for all $0 \leq j \leq k-1$. Finally, by using inequalities \eqref{eq:stageonehoelder} and \eqref{eq:simpleoscillation} repeatedly, we have that $$|u(x,t) - u(y,s)| \leq \omega_0 \leq \omega_1^+ \leq \zeta \omega_1 \leq \zeta \omega_2^+ \leq \zeta^2 \omega_2 \leq \ldots \leq \zeta^k\omega_k.$$ Let $\nu:=\log_4(\frac{1}{\zeta})$, so that $\zeta = (\frac{1}{4})^{\nu}$. Then, we rewrite the last inequality as
$$|u(x,t) - u(y,s)| \leq \left(\frac{1}{4}\right)^{k\nu} \omega_k = \left(\frac{4}{\delta}\right)^{\nu}\left(\frac{\delta}{4^{k+1}}\right)^{\nu}\omega_k.$$
Then, applying the left hand inequality in \eqref{eq:kinteger}, it follows that 
$$|u(x,t) - u(y,s)| \leq \left(\frac{4}{\delta}\right)^{\nu} (|x-y| + |t-s|^{1/2})^{\nu} \omega_k.$$ Finally, since $$\omega_k \leq 2 \sup_{D_{R_k}(z,\tau_k)} u \leq 2 \sup_{Q} u = 2\|u\|_{L^{\infty}(Q)},$$ we obtain $$|u(x,t) - u(y,s)| \leq 2  \left(\frac{4}{\delta}\right)^{\nu} \|u\|_{L^{\infty}(Q)}(|x-y| + |t-s|^{1/2})^{\nu},$$ which implies \eqref{eq:hoelder}.

Lastly, the result in the case when \eqref{eq:caseiihoelder} holds follows as a simple consequence of this inequality. Indeed,
\begin{align*}
    |u(x,t) - u(y,s)| &\leq 2\|u\|_{L^{\infty}(Q)} = 2\|u\|_{L^{\infty}(Q)} \frac{\delta^{\nu}}{\delta^{\nu}}\\
    &\leq 2\|u\|_{L^{\infty}(Q)} \frac{(|x-y| + |t-s|^{1/2})^{\nu}}{\delta^{\nu}}\\
    & \leq 2  \left(\frac{4}{\delta}\right)^{\nu} \|u\|_{L^{\infty}(Q)}(|x-y| + |t-s|^{1/2})^{\nu}
\end{align*}

Inequality \eqref{eq:hoelder} implies that $u$ is uniformly H\"older continuous on convex subsets $Q'$ of $Q$. To extend this result to all compact subsets of $Q$, we recall that any compact set in $Q$ can be covered by finitely many open balls. Since open balls are convex sets, we may apply \eqref{eq:hoelder} on each ball to obtain uniform H\"older continuity on compact sets in $Q$, or equivalently, local H\"older continuity on $Q$.

\end{proof}

We end this chapter by commenting that similar proofs have been found to demonstrate the local H\"older continuity of solutions of nonlinear equations, including the porous medium equation and $p$-diffusion equation studied in Chapter \ref{sec:applications}. However, we note that even though H\"older continuity often follows as a consequence of a Harnack inequality, the reverse implication is not true in general. As a counterexample, one may consider solutions of the $p$-diffusion equation pertaining to the subcritical range $1 < p \leq \frac{2d}{d+1}$. These solutions are known to be H\"older continuous, but as noted in Remark \ref{rem:pdiffsubcritical}, they do not satisfy a Harnack inequality similar to \eqref{eq:moserharnack} \cite{DiBenedettoGianazzaVespri}.

\chapter{Final Remarks}

\begin{remark}[A Note on the Assumptions of Theorems \ref{thm:generalharnackineq}, \ref{thm:gen-dist}, and \ref{thm:gen-backwards}]
In each of the general Harnack inequality theorems proven in Chapter \ref{sec:genharnack}, it was always assumed that the constant $C$ appearing in the gradient estimates \eqref{eq:thm1} and \eqref{eq:thm4} was strictly positive. Not only would the argument via Young's inequality in the proof break down if we allowed $C\leq 0$, such an assumption would not lead to a valid result. We demonstrate this by a counterexample.

Let $1 < p \leq \frac{2d}{d+1}$ and let $u$ be a positive solution of the $p$-diffusion equation $$u_t = \Delta_p u \qquad \text{in }\Omega_T,$$ where we take $\Omega$ to be any open convex set in $\R^d$. An inequality of B\'enilan and Crandall \cite{benilan-crandall} states that
\begin{equation}\label{eq:bc}
    u_t \leq -\frac{1}{(p-2)t}u \qquad \text{in }\Omega_T.
\end{equation}

Letting $g:=-\frac{1}{\gamma}u^{\gamma}$ with $\gamma = \frac{p-2}{p-1}$ and using a calculation from Section \ref{sec:pdiff}, we know $$\Delta_p(-g) = \frac{\Delta_p u}{u} - \frac{|\nabla u|^p}{u^2} = \frac{u_t}{u} - \frac{|\nabla u|^p}{u^2}.$$ Making use of \eqref{eq:bc}, we obtain that $$-\Delta_p g \leq -\frac{1}{(p-2)t} - \frac{|\nabla u|^p}{u^2} \leq -\frac{1}{(p-2)t}.$$ This leads to the inequality \begin{equation}\label{eq:pseudo-estvaz}
    \Delta_p g \geq \frac{1}{(p-2)t}.
\end{equation}
Recall that $g$ satisfies the equation  \begin{equation*}\tag{\ref{eq:newpdiffg}}\frac{\partial g}{\partial t} + \gamma g \Delta_p g = -|\nabla g|^p.\end{equation*} Combining this with \eqref{eq:pseudo-estvaz} gives the inequality $$\frac{\partial g}{\partial t} + \frac{\gamma}{(p-2)t}g \geq -|\nabla g|^p,$$ which appears to be of the form \eqref{eq:thm1}, except here $C=-1 <0$. However, Harnack inequality of the form discussed in Theorems \ref{thm:generalharnackineq} or \ref{thm:gen-dist} cannot follow from here, since this would also imply a Harnack inequality for $u$. This would contradict the result mentioned in Remark \ref{rem:pdiffsubcritical}, that a solution of the $p$-diffusion equation for $p$ in the subcritical range $1 < p \leq \frac{2d}{d+1}$ does not in general satisfy a Harnack inequality of this form \cite{DiBenedettoGianazzaVespri}. Moreover, inequality \eqref{eq:bc} can be considered as being of the form \eqref{eq:thm4} with $C=0$. However, a Harnack inequality cannot follow from \eqref{eq:bc} alone for this same reason.

\end{remark}

\begin{remark}[The Doubly Nonlinear Equation]
The porous medium equation $u_t = \Delta(u^M)$ and $p$-diffusion equation $u_t = \Delta_p u$ are both special cases of a more general equation, \begin{equation}\label{eq:dnle} u_t = \Delta_p(u^M)\end{equation} called the \emph{doubly nonlinear equation}. In particular, we may recover the porous medium equation from \eqref{eq:dnle} by setting $p=2$ and the $p$-diffusion equation by setting $M=1$.

In the paper of Esteban and V\'azquez \cite{Esteban-Vasquez-1990} containing Lemma \ref{lem:estvaz}, the authors concluded with a remark that their methods could be extended to prove an analogous result for nonnegative solutions of \eqref{eq:dnle} in $\R^d_T$ with $M$ and $p$ satisfying $M(p-1) - 1 + \frac{p}{d} >0$. Given the existence of this result, we believe it possible to derive a Harnack inequality satisfied by solutions of \eqref{eq:dnle} using the methods discussed in this monograph.

\end{remark}

\section{Future Work}

We conclude by addressing some potential directions for future work.
\vspace{15pt}

\hspace{-18pt}\textbf{Boundary Value Problems}
\vspace{6pt}

In Chapter \ref{sec:applications}, we applied the theorems of Auchmuty and Bao \cite{Auchmuty-Bao-1994} to derive Harnack inequalities for solutions of evolution problems posed on the domain $\R^d_T:= \R^d \times (0,T)$. Although the results of Auchmuty and Bao hold when the spatial domain is any open convex set $\Omega$ in $\R^d$, we were limited to choose $\Omega = \R^d$ in our applications. This is because the derivation of an appropriate gradient estimate of the form \eqref{eq:thm1} or \eqref{eq:thm4} depended fundamentally on the inequalities of Aronson and B\'enilan (Lemma \ref{lem:ab}) and Esteban and V\'azquez (Lemma \ref{lem:estvaz}) in the cases of the porous medium equation and $p$-diffusion equation respectively. To our knowledge, these estimates are only known to hold for solutions of their respective problems posed on the full spatial domain $\R^d$. In order to extend the techniques demonstrated in this monograph to find Harnack inequalities for solutions of boundary value problems, one would require a suitable analogue of these results, which is valid for such solutions.
\vspace{15pt}

\hspace{-18pt}\textbf{The $p$-Dirichlet-to-Neumann Operator}
\vspace{6pt}

We are interested to apply the techniques developed in this monograph to discover Harnack inequalities obeyed by solutions of other nonlinear evolution problems. In particular, we would like to investigate whether solutions of the parabolic problem associated with the $p$-Dirichlet-to-Neumann operator satisfy such an inequality. We briefly describe the construction of this operator below, which may be found in \cite{Hauer-JDE} along with a discussion of its basic properties.

Let $\Omega$ be a bounded domain in $\R^d$ with a Lipschitz continuous boundary. Then, it is well known that for $1<p<\infty$ and for all boundary values $\varphi \in W^{1-1/p,p}(\partial \Omega)$, there exists a unique weak solution of the $p$-Dirichlet problem
\begin{equation*}
\begin{cases}\label{eq:p-Dirichlet}
-\Delta_p u=0 & \text{in } \Omega, \\
u=\varphi  & \text{on } \partial \Omega.
\end{cases}
\end{equation*}
Denoting this weak solution by $P\varphi$, we formally define the \emph{Dirichlet-to-Neumann operator} $\Lambda$ associated with the $p$-Laplace operator $\Delta_p$ by $$\Lambda \varphi:=|\nabla P\varphi|^{p-2} \frac{\partial P\varphi}{\partial \nu}$$ for all $\varphi \in W^{1-1/p,p}(\Omega)$, where $\nu$ is the outward pointing unit normal vector on the boundary $\partial \Omega$.

With reference to the physical interpretation of the $p$-Laplace operator $\Delta_p$ given in Section \ref{sec:modelling}, we may understand the Dirichlet-to-Neumann operator $\Lambda$ as mapping the electric potential on the boundary of a medium $\Omega$ to the outward pointing  current through the boundary $\partial \Omega$. This operator appears in inverse problems associated with the $p$-Laplace operator and has applications, for example, to medical imaging. In particular, the operator $\Lambda$ can be used to learn about the composition of the medium $\Omega$ by detecting regions of various levels of conductivity. For instance, the location of bones within a body can be identified due to their lower conductivity compared with surrounding body tissues.

An important property of the operator $\Lambda$ is that unlike the other operators discussed in this monograph, it is a nonlocal operator. By this, we mean that the value of the function $\Lambda \varphi$ cannot be determined at a point only using the values of $\varphi$ in a neighbourhood of that point. Instead, the function values of $\Lambda \varphi$ can only be determined if the value of $\varphi$ is known on its entire domain. At its core, a Harnack inequality provides a uniform estimate of a function in a neighbourhood of a point given the value of the function at just that one point. Hence, it can be understood as a local property of a function. Thus, we are interested to see whether such an inequality could hold for functions whose evolution is governed by the  Dirichlet-to-Neumann operator, given its nonlocal properties.


\newpage

\appendix
\renewcommand{\thesection}{\Alph{chapter}.\arabic{section}}
\chapter{The Sobolev Spaces $W^{1,p}(a,b)$ and $W^{1,p}(\Omega)$ }\label{sobolev}
Here, we define and state some basic properties of the Sobolev spaces $W^{1,p}(a,b)$ and $W^{1,p}(\Omega)$.
The results in this section will be primarily taken from Brezis \cite{Brezis}. Accordingly, we adopt the same convention used in \cite{Brezis}, whereby $\int_a^b f$ will be used to denote the integral $\int_a^b f(x) \ \mathrm{d}\mu(x)$ with respect to the Lebesgue measure.

\section{The Sobolev Space $W^{1,p}(a,b)$}
Let $(a,b)$ be a (possibly unbounded) open interval. For every $p \in \R$ with $1 \leq p \leq \infty$, the Sobolev space $W^{1,p}(a,b)$ is the set of all functions $u \in L^p(a,b)$ for which there exists another function $g \in L^p(a,b)$ such that $$\int_a^b u\varphi' = -\int_a^b g\varphi$$ for all $\varphi \in C^\infty_c(a,b)$. Moreover, we denote the space $W^{1,2}(a,b)$ by $H^1(a,b)$.

For $u \in W^{1,p}(a,b)$, the function $u':=g$ is called the weak derivative of $u$. If a function $u$ has a weak derivative, then this derivative is unique. This is a consequence of the Fundamental Lemma of Calculus of Variations.

\begin{lemma}[Fundamental Lemma of Calculus of Variations]\label{fundlemma}
Let $\Omega \subseteq \R^d$ be open and let $u \in L^1_{\loc} (\Omega)$ be such that $$\int_{\Omega} u\varphi = 0$$ for all $\varphi \in C_c^\infty (\Omega)$. Then $u=0$ a.e. on $\Omega$.
\end{lemma}

We now collect some basic properties of the space $W^{1,p}(a,b)$.
\begin{proposition}
The space $W^{1,p}(a,b)$ satisfies the following:
\begin{enumerate}[label=(\roman*)]
    \item For $1 \leq p \leq \infty$, $W^{1,p}(a,b)$ with the norm $$\|u\|_{W^{1,p}}:= \|u\|_{p} + \|u'\|_{p}$$ is a Banach space. Alternatively, one may also define the norm by $\|u\|_{W^{1,p}}:= (\|u\|_{p}^p + \|u'\|_{p}^p)^{1/p}$ and the result still holds;
    \item $W^{1,p}(a,b)$ is reflexive for $1 < p < \infty$;
    \item $W^{1,p}(a,b)$ is separable for $1 \leq p < \infty$.
\end{enumerate}
\end{proposition}
It follows from this proposition that $H^1(a,b)$ is a separable Hilbert space with the inner product
$$(u,v)_{H^1}:= (u,v)_{L^2} + (u',v')_{L^2} = \int_a^b (uv + u'v')$$
and induced norm
$$\|u\|_{H^1} = (\|u\|^2_{2} + \|u'\|^2_{2})^{1/2}.$$

In general, there is no requirement that members of a Sobolev space be continuous functions. This is especially true for Sobolev spaces of functions on higher-dimensional domains, which will be defined later in Appendix \ref{sec:multidimsobolev}. However, in the one-dimensional case, a function $u \in W^{1,p}(a,b)$ has a \emph{continuous representative} $\tilde{u}$, which is described below in Theorem \ref{thmcontrep}. The existence of such a continuous representative can be rather useful, especially for problems where continuity is necessary. In most cases, the function $\tilde{u}$ is identified with $u$ and no distinction is made between them in notation.

\begin{theorem}\label{thmcontrep}
Suppose $u \in W^{1,p}(a,b)$. Then, there exists a function \mbox{$\tilde{u} \in C([a,b])$} such that $u = \tilde{u}$ a.e. on $(a,b)$ and $$
\tilde{u}(x) - \tilde{u}(y) = \int_y^x u'(t) \ \mathrm{d}t
$$ for all $x, y \in [a,b]$.
\end{theorem}

The proof of this theorem, as found in \cite{Brezis}, makes use of the following two lemmas.

\begin{lemma}\label{82lemma1}
Let $f \in L^1_{\loc} (a,b)$ be such that
\begin{equation*}
    \int_a^b f \varphi' = 0
\end{equation*} for all  $\varphi \in C_c^\infty (a,b)$. Then there exists a constant $C$ such that $f=C$ a.e. on $(a,b)$.
\end{lemma}

\begin{lemma}\label{82lemma2}
Let $g \in L^1_{\loc}(a,b)$. For $y_0 \in (a,b)$ fixed, set
$$v(x):=\int_{y_0}^x g(t) \ \mathrm{d}t$$ for all $x \in (a,b)$. Then $v \in C(a,b)$ and $$\int_a^b v \varphi' = -\int_a^b g \varphi$$ for all  $\varphi \in C_c^\infty (a,b)$.
\end{lemma}

Some familiar rules from classical calculus have analogues for Sobolev spaces.
\begin{theorem}
Let $u,v \in W^{1,p}(a,b)$ with $1 \leq p \leq \infty$. Then $uv \in W^{1,p}(a,b)$ with weak derivative given by $$(uv)' = u'v + uv'.$$ The following integration by parts formula also holds $$\int_x^y u'v = u(x)v(x) -u(y)v(y) - \int_y^x uv' \qquad \text{for all } x,y \in [a,b].$$
\end{theorem}

We now define $W^{1,p}_0(a,b)$ to be the closure of $C_c^{\infty}(a,b)$ in $W^{1,p}(a,b)$ for $1 \leq p < \infty$. We also denote the space $W^{1,2}_0(a,b)$ by $H^1_0(a,b)$. The spaces $W^{1,p}_0(a,b)$ enjoy many of the same fundamental properties as the spaces $W^{1,p}(a,b)$, specifically completeness and separability for $p\geq 1$ and reflexivity for $p > 1$. 

The space $W^{1,p}_0(a,b)$ can be understood via the following characterisation.

\begin{theorem}\label{w0char}
Let $u \in W^{1,p}(a,b)$. Then $u \in W^{1,p}_0(a,b)$ if and only if \newline $u(a)=u(b)=0$.
\end{theorem}

\section{The Sobolev Space $W^{1,p}(\Omega)$ }\label{sec:multidimsobolev}
Let $\Omega \subseteq \R^d$ be an open set and $1 \leq p \leq \infty$. The Sobolev space $W^{1,p}(\Omega)$ is defined as the set of functions $u \in L^p(\Omega)$ for which there exist functions $g_1,g_2,\ldots,g_d \in L^p (\Omega)$ such that
$$\int_{\Omega}u \frac{\partial \varphi}{\partial x_i} = -\int_{\Omega}g_i\varphi$$
for all $\varphi \in C_c^{\infty}(\Omega)$ and for all $i=1,2,\ldots,d$. For a function $u\in W^{1,p}(\Omega)$, we write $\frac{\partial u}{\partial x_i}:=g_i$ and $$\nabla u := \left( \frac{\partial u}{\partial x_1},\frac{\partial u}{\partial x_2},\ldots,\frac{\partial u}{\partial x_d} \right) \in (L^p(\Omega))^d.$$

Some of the basic properties of $W^{1,p}(\Omega)$ are as follows.

\begin{proposition}
The space $W^{1,p}(\Omega)$ satisfies the following:
\begin{enumerate}[label=(\roman*)]
    \item For $1 \leq p \leq \infty$, $W^{1,p}(\Omega)$ with the norm $$\|u\|_{W^{1,p}}:= \|u\|_{p} + \sum_{i=1}^d\left\|\frac{\partial u}{\partial x_i}\right\|_{p}$$ is a Banach space. Alternatively, one may also define the norm by $\|u\|_{W^{1,p}}:= (\|u\|_{p}^p + \sum_{i=1}^d\|\frac{\partial u}{\partial x_i}\|_{p}^p)^{1/p}$ and the result still holds.
    \item $W^{1,p}(\Omega)$ is reflexive for $1 < p < \infty$.
    \item $W^{1,p}(\Omega)$ is separable for $1 \leq p < \infty$.
\end{enumerate}
\end{proposition}

We also define the space $W^{1,p}_0(\Omega)$ as the closure of the test functions $C_c^{\infty}(\Omega)$ in $W^{1,p}(\Omega)$.

\chapter{Minimisation of Convex Functionals}\label{sec:convexmin}
Let $V$ be a vector space and $E:V\rightarrow ( -\infty, + \infty]$. The functional $E$ is called \emph{convex} if for all $x,y \in V$ and $\lambda \in [0,1]$,
$$E(\lambda x + (1-\lambda)y) \leq \lambda E(x) + (1-\lambda) E(y).$$
If the inequality is strict for $\lambda \in (0,1)$, then $E$ is called \emph{strictly convex}.

We are interested in solving minimisation problems of the form
$$\min_{x \in V} E(x),$$
where $(V, \|\cdot\|_V)$ is a reflexive Banach space and $E:V\rightarrow ( -\infty, + \infty]$ is a convex functional. To this end, we briefly discuss some other assumptions that must be placed on the functional $E$ to ensure the existence of a minimum, namely coercivity and lower semicontinuity. These results shall be taken from \cite{Attouch}.

A functional $E:V \rightarrow ( -\infty, + \infty)$ is called \emph{coercive} if $$\lim_{\|x\|_{V}\rightarrow \infty} E(x) = +\infty.$$ 

\begin{proposition}\label{coercivebdd}
Let $V$ be a normed vector space and $E:V \rightarrow ( -\infty, + \infty]$. Then $E$ is coercive if and only if for every $c \in \R$, the \emph{sublevel set}
$$E_c := \{x \in V \mid E(x) \leq c\}$$ is bounded.
\end{proposition}

A functional $E:V \rightarrow ( -\infty, + \infty]$ is called  \emph{(sequentially) lower semicontinuous} on $V$ if for all $x \in V$ and for all sequences $(x_n)_n$ in $V$ such that $x_n \rightarrow x$ as $n\rightarrow \infty$, one has that $$E(x) \leq \liminf_{n \rightarrow \infty} E(x_n).$$

\begin{proposition}\label{weaklylsc}
Let $E:V \rightarrow ( -\infty, + \infty]$ be a proper convex lower semicontinuous functional. Then $E$ is weakly lower semicontinuous on $V$, that is, if $x_n \rightharpoonup x$ in $V$ as $n\rightarrow \infty$, then $$E(x) \leq \liminf_{n \rightarrow \infty} E(x_n).$$
\end{proposition}

More general minimisation principles often rely on the space $V$ having some sense of compactness. In our context, $V$ will be a reflexive Banach space, so this requirement is handled by the following result.

\begin{proposition}\label{bddseqref}
Every bounded sequence in a reflexive space $V$ contains a weakly convergent subsequence.
\end{proposition}

We are now ready to state the main theorem of this section.

\begin{theorem}\label{convexminthm}
Suppose that $(V, \|\cdot\|_V)$ is a reflexive Banach space and let $E:V \rightarrow ( -\infty, + \infty)$ be a proper convex, lower semicontinuous and coercive functional. Then there exists $x_0 \in V$ such that $$E(x_0) = \min_{x \in V} E(x).$$ Moreover, if $E$ is strictly convex, then the minimiser $x_0$ is unique.
\end{theorem}

After establishing the existence of a minimum of a convex functional $E$, we would like to determine for which element(s) $x \in V$ this minimum is attained. In order to characterise the minimiser(s) $x$, we introduce the subdifferential of $E$.

For a functional $E:V \rightarrow (-\infty, \infty]$, the \emph{subdifferential} of $E$ at a point $x \in V$ is defined as $$\partial E(x) := \{ v' \in V' \mid E(y)-E(x) \geq \langle v',y-x\rangle_{V',V} \text{ for all } y \in V\},$$ 
where $V'$ denotes the dual space of $V$. The following proposition is a simple consequence of this definition.
\begin{proposition}
\label{subdiffmin}
Let $E:V \rightarrow (-\infty, +\infty]$ be a functional. Then a point $x \in V$ minimises $E$ if and only if $0 \in \partial E(x)$.
\end{proposition}

The subdifferential of a functional $E$ at $x$ may in general contain more than one point, or even be empty. However, if $E$ is G\^ateaux differentiable at $x$, then $\partial E(x)$ becomes single valued \cite{Showalter}.

\begin{proposition}\label{subdiffgateaux}
Let $E: (-\infty,+\infty]$ be convex and G\^ateaux differentiable at $x$. Then $\partial E(x)$ is a singleton and $\partial E(x) = \{ E'(x)\}$, where $E'(x)$ denotes the G\^ateaux derivative of $E$ at $x$ as defined in Chapter \ref{sec:min}.
\end{proposition}

Combining the results of Propositions \ref{subdiffmin} and \ref{subdiffgateaux}, we have that if $E$ is convex and G\^ateaux differentiable at its minimiser $x$, then it is forced that $E'(x)=0$. This is analogous to the equivalent statement from classical calculus.

\chapter{Absolutely Continuous Functions}\label{sec:abs-cont}
A crucial assumption in Theorems \ref{thm:gen-dist} and \ref{thm:gen-backwards} is that the function $f$ is absolutely continuous. Hence, we summarise some basic properties and results involving absolutely continuous functions. Unless noted otherwise, the following may be found in \cite{Leoni}.

Let $I \subseteq \R$ be an interval. A function $u: I \rightarrow \R^d$ is called \emph{absolutely continuous} on $I$ if for every $\varepsilon > 0$ there exists $\delta >0$ such that $$\sum_{i=1}^n |u(b_i) - u(a_i)| \leq \varepsilon$$ for every finite number of non-overlapping intervals $(a_i,b_i)$, $i=1, \ldots, n$, such that $[a_i,b_i] \subseteq I$ and $$\sum_{i=1}^n (b_i - a_i) \leq \delta. $$ Equivalently, we may replace $n$ by $\infty$ in this definition. We also note that by taking $n=1$ in this definition, it follows immediately that an absolutely continuous function $u$ is uniformly continuous on $I$. However, the converse is not true in general. Similarly, it is an immediate consequence of the definition that any Lipschitz continuous function on $I$ is also absolutely continuous on $I$.

The next propositions provide some frequently used examples of absolutely continuous functions.

\begin{proposition}\label{prop:basic-abs}
Let $u,v:I \rightarrow \R$ be absolutely continuous on a bounded interval $I \subseteq \R$. Then, $u \pm v$ and $uv$ are absolutely continuous on $I$. If $v$ is positive on $I$, then $\tfrac{u}{v}$ is also absolutely continuous on $I$.
\end{proposition}

Although the composition of two absolutely continuous functions is not absolutely continuous in general, we do have the following result.

\begin{proposition}\label{prop:lip-abs}
Let $I \subseteq \R$ be an interval and $u:I \rightarrow \R^d$ be absolutely continuous on $I$. If $f: \R^d \rightarrow \R$ is Lipschitz continuous, then $f \circ u$ is absolutely continuous on $I$.
\end{proposition}

\begin{proposition}\label{prop:prim-abs}
Let $I \subset \R$ be an interval and $v:I \rightarrow \R^d$ be a Lebesgue integrable function. Fix $x_0 \in I$ and set $$u(x):= \int_{x_0}^x v(t) \ \mathrm{d}t$$ for all $x \in I$. Then $u$ is absolutely continuous on $I$ with $u'(x) = v(x)$ for a.e. $x \in I$.
\end{proposition}

A main motivation for introducing the notion of absolute continuity is that it characterises the set of functions for which the fundamental theorem of calculus holds for Lebesgue integration.

\begin{theorem}[Fundamental Theorem of Calculus, \cite{Leoni}]\label{thm:ftc-abs}
Suppose a function $u : I \rightarrow \R^d$ is absolutely continuous on $I$. Then 
\begin{enumerate}
\item $u$ is continuous in $I$;
\item $u$ is differentiable $\mathcal{L}^1$-a.e. in $I$ with derivative $u' \in L^1_{\loc}(I, \R^d)$;
\item the fundamental theorem of calculus holds, that is, $$u(x) = u(x_0) + \int_{x_0}^xu'(t) \ \mathrm{d}t$$ for all $x,x_0 \in I$.
\end{enumerate}
Conversely, if a function $u:I \rightarrow \R^d$ satisfies conditions $(i)-(iii)$, then $u$ is absolutely continuous on $I$.
\end{theorem}

The concept of absolute continuity also allows one to state the chain rule under weaker hypotheses than are required in classical calculus. In order to describe this result precisely, we first introduce the $s$-dimensional Hausdorff measure on $\R^d$ as defined in \cite{Evans-Gariepy}. This measure is particularly useful for measuring sets of lower dimension in $\R^d$, which have Lebesgue measure zero.

Let $A \subseteq \R^d$, $0 \leq s < \infty$, $0 < \delta \leq \infty$, and $$\mathcal{H}^s_{\delta}(A):= \inf \left\{ \sum_{i=1}^{\infty} \omega_s \Big(\tfrac{\diam C_i}{2}\Big)^s \mid A \subseteq \bigcup_{i=1}^\infty C_i, \ \diam C_i \leq \delta \right\}$$ where $\omega_s:= \tfrac{\pi^{s/2}}{\Gamma(1+\tfrac{s}{2})}$ is the Lebesgue measure of the unit ball in $\R^s$ and $$\diam C_i:= \sup_{x,y \in C_i} |x-y|$$ is the diameter of the set $C_i$. Then, define $$\mathcal{H}^s(A):= \lim_{\delta \rightarrow 0} \mathcal{H}^s_{\delta}(A) = \sup_{\delta >0} \mathcal{H}^s_{\delta} (A).$$ We call $\mathcal{H}^s$ the \emph{s-dimensional Hausdorff measure on $\R^d$}. In the case $s=d$, the $d$-dimensional Hausdorff measure on $\R^d$ coincides with the Lebesgue measure on $\R^d$.

In addition, we say that a set $A \subseteq \R^d$ has the \emph{null intersection property} if the intersection of $S$ with the image of any absolutely continuous curve $u:I \rightarrow \R^d$ is a set of $\mathcal{H}^1$ measure zero. We now state a version of the chain rule holding under weaker assumptions, which is originally due to Marcus and Mizel \cite{Marcus-Mizel}.

\begin{theorem}[Chain Rule]\label{thm:chain-abs}
Let $u:I \rightarrow \R^d$ be absolutely continuous and denote the image of $u$ by $T_u:=u(I)$. Let $f:\R^d \rightarrow \R$ be a function such that 
\begin{enumerate}
    \item the set of points at which $f$ does not have a total derivative has the null intersection property;
    \item $\nabla f(u) \cdot \dot{u}$ satisfies $$\int_I |\nabla f(u) \cdot \dot{u}| \ \mathrm{d}t < \infty;$$
    \item $f$ is absolutely continuous on all continuously differentiable curves $u:I \rightarrow \R^d$.
\end{enumerate}
Then $f|_{T_u}$ is absolutely continuous on $T_u$. In addition, the composition $g:=f \circ u$ is absolutely continuous on $I$ and the chain rule holds, that is, $$\dot{g} = \nabla f(u) \cdot \dot{u} \qquad \text{a.e. in } I.$$
Here, the expression $\nabla f(u) \cdot \dot{u} = \nabla f(u(t)) \cdot \dot{u}(t)$ is interpreted to be zero whenever $\dot{u} = 0$, even if $f$ is not differentiable at $u(t)$.
\end{theorem}

We remark that it was shown in \cite{Marcus-Mizel} that any locally Lipschitz continuous function $f: \R^d \rightarrow \R$ satisfies assumptions (i) and (iii) of Theorem \ref{thm:chain-abs}.

\chapter{Distribution Theory}\label{sec:distributions}

We briefly summarise the most important notions and examples from the theory of distributions \cite{Adams-Fournier}, which we use throughout this monograph.

Let $\Omega$ be a domain in $\R^d$. We say that a sequence $(\varphi_n)_{n\geq 1}$ in $C^{\infty}_c(\Omega)$ converges \emph{in the $\mathscr{D}(\Omega)$ sense} to $\varphi \in C_c^{\infty}(\Omega)$ if the two following conditions hold:
\begin{enumerate}
    \item there exists $K\subset \subset \Omega$ such that $\supp(\varphi_n - \varphi) \subset K$ for every $n \geq 1$;
    \item $\lim_{n \rightarrow \infty}D^{\alpha} \varphi_n(x) = D^{\alpha}\varphi(x)$ uniformly on $K$, for each multi-index $\alpha \in \N_0^d$, where $$D^{\alpha}:=\frac{\partial^{|\alpha|}}{\partial x_1^{\alpha_1}\partial x_2^{\alpha_2}\ldots\partial x_d^{\alpha_d}}$$ and $|\alpha|:=\sum_{i=1}^{d} \alpha_i$.
\end{enumerate}
Then, we denote by $\mathscr{D}(\Omega)$ the set of test functions $C_c^{\infty}(\Omega)$ equipped with the locally convex topology $\tau$, for which every linear functional \mbox{$T:\mathscr{D}(\Omega) \rightarrow \R$ } is continuous if and only if $T\varphi_n \rightarrow T\varphi$ in $\R$ whenever $\varphi_n \rightarrow \varphi$ in the $\mathscr{D}(\Omega)$ sense. The dual space $(\mathscr{D}(\Omega))'$ of $\mathscr{D}(\Omega)$ is called the space of (Schwartz) distributions on $\Omega$ and denoted by $\mathscr{D}'(\Omega)$.

Important examples of distributions include locally integrable functions, since every $f \in L^1_{\loc}(\Omega)$ induces a distribution $T_f \in \mathscr{D}'(\Omega)$ defined by
$$\langle T_f, \varphi \rangle_{\mathscr{D}',\mathscr{D}} = \int_{\Omega} f(x) \varphi(x) \ \mathrm{d}x$$ for all $\varphi \in \mathscr{D}(\Omega)$. If there exists $f\in L^1_{\loc}(\Omega)$, which induces a given distribution $T_f\in \mathscr{D}'(\Omega)$, then $T_f$ is called a \emph{regular distribution}. Often a regular distribution $T_f$ and its associated function $f$ are identified with each other and no distinction is made in notation. We note that not every distribution is regular. For example, assuming $0 \in \Omega$, there is no locally integrable function which induces the Dirac delta function $\delta \in \mathscr{D}'(\Omega)$ defined by $\langle \delta,\varphi \rangle_{\mathscr{D}',\mathscr{D}}:=\varphi(0)$.

Furthermore, given a distribution $T \in \mathscr{D}'(\Omega)$, we may define the \emph{distributional derivatives} of $T$ by $$\langle \frac{\partial}{\partial x_i}T,\varphi \rangle_{\mathscr{D}',\mathscr{D}}:=-\langle T,\frac{\partial \varphi}{\partial x_i}\rangle_{\mathscr{D}',\mathscr{D}}$$ for all $i=1,\ldots,d$. If $T$ is a regular distribution induced by $f \in L^{1}_{\loc}(\Omega)$ and $f$ has a weak derivative $g:=\frac{\partial f}{\partial x_i} \in L^1_{\loc}(\Omega)$, then $\frac{\partial}{\partial x_i}T$ is the distribution corresponding to $g$. Using this notion, one may understand the Sobolev space $W^{1,1}_{\loc}(\Omega)$ as the space of regular distributions $f$ such that the distributional derivatives $\frac{\partial f}{\partial x_i}$ are all regular distributions.

\backmatter
\bibliographystyle{plain}

\begin{thebibliography}{10}

\bibitem{Adams-Fournier}
R.~A. Adams.
\newblock {\em Sobolev Spaces}.
\newblock Academic Press, 1975.

\bibitem{Alikakos}
N.~D. Alikakos and R.~Rostamian.
\newblock Gradient estimates for degenerate diffusion equations {II}.
\newblock {\em Proceedings of the Royal Society of Edinburgh}, 91(A):335--346,
  1982.

\bibitem{Aronson-Benilan-1979}
D.~G. Aronson and P.~B{\'e}nilan.
\newblock R{\'e}gularit{\'e} des solutions de l'{\'e}quation des milieux poreux
  dans $\mathbb{R}^n$.
\newblock {\em C. R. Acad. Sci. Paris S{\'e}r. I Math.}, 288:103--105, 1979.

\bibitem{Attouch}
H.~Attouch, G.~Buttazzo, and G.~Michaille.
\newblock {\em Variational Analysis in Sobolev and BV Spaces}.
\newblock Society for Industrial and Applied Mathematics, 2005.

\bibitem{Auchmuty-Bao-1994}
G.~Auchmuty and D.~Bao.
\newblock Harnack-type inequalities for evolution equations.
\newblock {\em Proc. Amer. Math. Soc.}, 122(1):117--129, 1994.

\bibitem{Axler}
S.~Axler, P.~Bourdon, and W.~Ramey.
\newblock {\em Harmonic Function Theory}.
\newblock Springer-Verlag, 1992.

\bibitem{Barbu}
V.~Barbu.
\newblock {\em Nonlinear Differential Equations of Monotone Types in {B}anach
  Spaces}.
\newblock Springer Monographs in Mathematics. Springer, 2010.

\bibitem{benilan-crandall}
P.~B\'enilan and M.~G. Crandall.
\newblock {\em Regularizing effects of homogeneous evolution equations}.
\newblock Johns Hopkins Univ. Press, 1981.

\bibitem{Bonforte}
M.~Bonforte, J.~Dolbeault, B.~Nazaret, and N.~Simonov.
\newblock {Explicit constants in {H}arnack inequalities and regularity
  estimates, with an application to the fast diffusion equation}.
\newblock Preprint available at
  \url{https://hal.archives-ouvertes.fr/hal-02887013}, July 2020.

\bibitem{Brander}
T.~Brander, J.~Ilmavirta, and M.~Kar.
\newblock Superconductive and insulating inclusions for linear and non-linear
  conductivity equations.
\newblock {\em Inverse problems and imaging}, 12(1):91--123, 2018.

\bibitem{Brezis}
H.~Brezis.
\newblock {\em Functional Analysis, Sobolev Spaces and Partial Differential
  Equations}.
\newblock Springer, 2010.

\bibitem{Coulhon-Hauer}
T.~Coulhon and D.~Hauer.
\newblock Functional inequalities and regularising properties of nonlinear
  semigroups: Theory and applications.
\newblock To appear in BCAM Springer Briefs, preprint available at
  \url{http://arxiv.org/abs/1604.08737}, 2016.

\bibitem{DiBenedetto-Friedman-1985}
E.~DiBenedetto and A.~Friedman.
\newblock H\"older estimates for nonlinear degenerate parabolic systems.
\newblock {\em Journal f\"ur Mathematik}, 357:1--22, 1985.

\bibitem{DiBenedettoGianazzaVespri}
E.~DiBenedetto, U.~Gianazza, and V.~Vespri.
\newblock {\em Harnack's inequality for Degenerate and Singular Parabolic
  Equations}.
\newblock Springer-Verlag, 2012.

\bibitem{Esteban-Vazquez-1988}
J.~R. Esteban and J.~L. V\'azquez.
\newblock Homogeneous diffusion in $\mathbb{R}$ with power-like nonlinear
  diffusivity.
\newblock {\em Archive for Rational Mechanics and Analysis}, 103:39--80, 1988.

\bibitem{Esteban-Vasquez-1990}
J.~R. Esteban and J.~L. V\'azquez.
\newblock Régularité des solutions positives de l'équation parabolique
  $p$-laplacienne.
\newblock {\em C. R. Acad. Sci. Paris}, 310(1):105--110, 1990.

\bibitem{Evans}
L.~C. Evans.
\newblock {\em Partial Differential Equations}, volume~19 of {\em Graduate
  Studies in Mathematics}.
\newblock American Mathematical Society, 1998.

\bibitem{Evans-Gariepy}
L.~C. Evans and R.~F. Gariepy.
\newblock {\em Measure Theory and Fine Properties of Functions}.
\newblock CRC Press LLC, 2015.

\bibitem{GilbargTrudinger}
D.~Gilbarg and N.~S. Trudinger.
\newblock {\em Elliptic Partial Differential Equations of Second Order}, volume
  224 of {\em Grundlehren der mathematischen Wissenschaften}.
\newblock Springer-Verlag, 1998.

\bibitem{Hadamard1954}
J.~Hadamard.
\newblock Extension \`a l'\'{e}quation de la chaleur d'un th\'{e}or\`eme de
  {A}. {H}arnack.
\newblock {\em Rend. Circ. Mat. Palermo (2)}, 3:337--346, 1954.

\bibitem{Harnack1887}
C.-G.~A. Harnack.
\newblock {\em Die Grundlagen der Theorie des logaritmischen Potentiales und
  der eindeutigen Potentialfunktion in der Ebene}.
\newblock Teubner, Leipzig, 1887.

\bibitem{Hauer-JDE}
D.~Hauer.
\newblock The $p$-{D}irichlet-to-{N}eumann operator with applications to
  elliptic and parabolic problems.
\newblock {\em Journal of Differential Equations}, 259(8):3615--3655, 2015.

\bibitem{Kellogg1929}
O.~D. Kellogg.
\newblock {\em Foundations of Potential Theory}.
\newblock Grundlehren der mathematischen Wissenschaften. Springer, Berlin,
  Germany, 1929.

\bibitem{Kemper1972}
J.~T. Kemper.
\newblock A boundary {H}arnack principle for {L}ipschitz domains and the
  principle of positive singularities.
\newblock {\em Commun. Pure Appl. Math}, 25:247--255, 1972.

\bibitem{KrylovSafonov1981}
N.~V. Krylov and M.~V. Safonov.
\newblock A certain property of solutions of parabolic equations with
  measurable coefficients.
\newblock {\em Mathematics of the USSR. Izvestiya}, 16(1):151--164, 1981.

\bibitem{Leoni}
G.~Leoni.
\newblock {\em A First Course in {S}obolev Spaces}, volume 181 of {\em Graduate
  Studies in Mathematics}.
\newblock American Mathematical Society, 2017.

\bibitem{LiYau1986}
P.~Li and S.~T. Yau.
\newblock On the parabolic kernel of the {S}chr\"odinger operator.
\newblock {\em Acta Math.}, 156:153--201, 1986.

\bibitem{Marcus-Mizel}
M.~Marcus and V.~J. Mizel.
\newblock Absolute continuity on tracks and mappings of {S}obolev spaces.
\newblock {\em Arch. Rational Mech. Anal}, 45:294--320, 1972.

\bibitem{Moser1964}
J.~Moser.
\newblock A {H}arnack inequality for parabolic differential equations.
\newblock {\em Commun. Pure Appl. Math}, 17:101--134, 1964.

\bibitem{Nash1958}
J.~Nash.
\newblock Continuity of solutions of parabolic and elliptic equations.
\newblock {\em American Journal of Mathematics}, 80(4):931--954, 1958.

\bibitem{Pini1954}
B.~Pini.
\newblock Sulla soluzione generalizzata di {W}iener per il primo problema di
  valori al contorno nel caso parabolico.
\newblock {\em Rend. Sem. Mat. Univ. Padova}, 23:422--434, 1954.

\bibitem{Safonov1983}
M.~V. Safonov.
\newblock Harnack's inequality for elliptic equations and the {H}\"older
  property of their solutions.
\newblock {\em Journal of Soviet Mathematics}, 21(5):851--863, 1983.

\bibitem{Showalter}
R.~E. Showalter.
\newblock {\em Monotone Operators in Banach Space and Nonlinear Partial
  Differential Equations}, volume~49 of {\em Mathematical Surveys and
  Monographs}.
\newblock American Mathematical Society, 1997.

\bibitem{VazquezPME}
J.~L. Vázquez.
\newblock {\em The Porous Medium Equation: Mathematical Theory}.
\newblock Oxford Mathematical Monographs. Oxford Science Publications, 2007.

\bibitem{Zeidler}
E.~Zeidler.
\newblock {\em Applied Functional Analysis: Main Principles and Their
  Applications}, volume 109 of {\em Applied Mathematical Sciences}.
\newblock Springer Verlag, 1995.

\end{thebibliography}

\end{document}